\documentclass[11pt,a4paper,leqno,twoside,notitlepage]{article}

\usepackage{amsmath}
\usepackage{amsfonts}
\usepackage{amssymb}
\usepackage{amsxtra}
\usepackage{amsthm}
\usepackage{enumerate}
\usepackage{ifthen}

\newboolean{details}

\newenvironment{enumerateroman}
{\begin{enumerate}}
{\end{enumerate}}

\numberwithin{equation}{section}

\theoremstyle{plain}
\newtheorem{thm}[equation]{Theorem}
\newtheorem{prop}[equation]{Proposition}
\newtheorem{lem}[equation]{Lemma}
\newtheorem{cor}[equation]{Corollary}

\theoremstyle{definition}
\newtheorem{defn}[equation]{Definition}
\newtheorem{rem}[equation]{Remark}

\newcommand{\wrt}{w.~r.~t.}

\newcommand{\mC}{{\mathbb C}}

\newcommand{\mR}{{\mathbb R}}

\newcommand{\mcA}{\mathcal{A}}
\newcommand{\mcB}{\mathcal{B}}
\newcommand{\mcC}{\mathcal{C}}

\newcommand{\mcK}{\mathcal{K}}

\newcommand{\mcP}{\mathcal{P}}

\newcommand{\mcS}{\mathcal{S}}
\newcommand{\mcU}{\mathcal{U}}
\newcommand{\fraka}{\mathfrak{a}}
\newcommand{\frakb}{\mathfrak{b}}

\newcommand{\frakg}{\mathfrak{g}}

\newcommand{\frakh}{\mathfrak{h}}
\newcommand{\frakH}{\mathfrak{H}}

\newcommand{\frakm}{\mathfrak{m}}
\newcommand{\frakn}{\mathfrak{n}}
\newcommand{\frakp}{\mathfrak{p}}
\newcommand{\frakq}{\mathfrak{q}}

\newcommand{\frakz}{\mathfrak{z}}

\newcommand{\LG}{L^1(G)}
\newcommand{\LN}{L^1(N)}

\newcommand{\LM}{L^1(M)}

\newcommand{\LH}{L^1(H)}

\newcommand{\CG}{C^{\ast}(G)}
\newcommand{\CN}{C^{\ast}(N)}
\newcommand{\CM}{C^{\ast}(M)}
\newcommand{\CH}{C^{\ast}(H)}
\newcommand{\Ccinfty}[1]{\mcC^{\infty}_0(#1)}

\newcommand{\Cc}[1]{\mcC_0(#1)}
\newcommand{\Cst}{C^{\ast}}

\newcommand{\twoheadlongrightarrow}{\,\,\rightarrow\mspace{-24.0mu}\longrightarrow\,}

\renewcommand{\to}{\longrightarrow}

\newcommand{\onto}{\mbox{$\twoheadlongrightarrow$}}

\newcommand{\mdot}{\!\cdot\!}
\newcommand{\clos}{^{\textbf{---}}}

\providecommand{\ind}{\mathop{\rm ind}\nolimits}
\providecommand{\Ad}{\mathop{\rm Ad}\nolimits}
\providecommand{\coAd}{\mathop{\rm Ad^{\ast}}\nolimits}

\providecommand{\Ann}{\mathop{\rm Ann}\nolimits}
\providecommand{\Lie}{\mathop{\rm Lie}\nolimits}
\providecommand{\Max}{\mathop{\rm Max}\nolimits}
\providecommand{\Prim}{\mathop{\rm Prim}\nolimits}

\providecommand{\lspan}{\mathop{\rm span}\nolimits}

\providecommand{\res}{\mathop{\rm res}\nolimits}

\setboolean{details}{true}

\begin{document}

\title{$L^1$-determined ideals in group algebras of
exponential Lie groups}
\author{O. Ungermann}
\date{}
\maketitle

\noindent\textbf{Abstract.} A locally compact group
$G$ is said to be $\ast$-regular if the natural map
$\Psi:\Prim\CG\to\Prim_{\ast}\LG$ is a homeomorphism
with respect to the Jacobson topologies on the primitive
ideal spaces $\Prim\CG$ and $\Prim_{\ast}\LG$. In 1980
J.~Boidol characterized the $\ast$-regular ones among
all exponential Lie groups by a purely algebraic
condition. In this article we introduce the notion
of $L^1$-determined ideals in order to discuss the
weaker property of primitive $\ast$-regularity. We
give two sufficient criteria for closed ideals $I$
of $\CG$ to be $L^1$-determined. Herefrom we deduce a
strategy to prove that a given exponential Lie group
is primitive $\ast$-regular. The author proved in his
thesis that all exponential Lie groups of dimension
$\le 7$ have this property. So far no counter-example
is known. Here we discuss the example $G=B_5$, the
only critical one in dimension $\le 5$.\\\\


\section{Introduction}
\label{eLg_sec:basic_definitions}
Let $\mcA$ be Banach $\ast$-algebra and $\Cst(\mcA)$ its
enveloping $\Cst$-algebra in the sense of Dixmier, see
Chapter 2.7 of~\cite{Dix3}. The $\Cst$-norm
on $\Cst(\mcA)$ is given by
$$|a|_{\ast}=\sup\limits_{\pi\in\widehat{\mcA}}\;|\pi(a)|$$
for all $a\in\mcA$ where $\widehat{\mcA}$ is the set of
equivalence classes of topologically irreducible
$\ast\,$-representations of $\mcA$ in Hilbert spaces.
Let $\Prim\Cst(\mcA)$ be the set of primitive ideals in
$\Cst(\mcA)$, and $\Prim_{\ast}\mcA$ the set of kernels
of representations in $\widehat{\mcA}$. For ideals $I$
of $\Cst(\mcA)$ we define their hull
$h(I)=\{P\in\Prim\Cst(\mcA):P\supset I\}$ in
$\Prim\Cst(\mcA)$, and for subsets $X$ of $\Prim\Cst(\mcA)$
their kernel $k(X)=\cap\{P:P\in X\}$ in $\Cst(\mcA)$. In
the sequel all ideals are assumed to be two-sided and
closed in the respective norm. Closed ideals $I$ of
$\Cst$-algebras are automatically involutive and satisfy
$I=k(h(I))$, see Proposition~1.8.2 and Theorem~2.6.1
of~\cite{Dix3}.\\\\
Recall that $\Prim\Cst(\mcA)$ is a topological space \wrt\
the Jacobson topology, i.e., $X\subset\Prim\Cst(\mcA)$ is
closed if and only if there exists an ideal $I$ of
$\Cst(\mcA)$ such that $X=h(I)$. Likewise we can state
the according definitions of hulls and kernels for $\mcA$
and we provide $\Prim_{\ast}\mcA$ with the Jacobson
topology as well. Let $I'$ denote the preimage of the
ideal $I$ under the natural map $\mcA\to\Cst(\mcA)$.
For simplicity we write $I'=I\cap\mcA$. The map
$$\Psi:\Prim\Cst(\mcA)\to\Prim_{\ast}\mcA\;%
\text{ given by }\;\Psi(P)=P'=P\cap\mcA$$
is continuous and surjective and evidently
satisfies $k(\Psi(X))=k(X)\cap\mcA$ and
$h(I)\subset\Psi^{-1}(h(I'))$. The next definition
is basic for the subsequent investigation.

\begin{defn}\label{Bsa_def:A_determined}
A closed ideal $I$ of $\Cst(\mcA)$ is called
$\mcA$-determined if and only if the following
(obviously) equivalent conditions hold:
\begin{enumerateroman}
\item $\;I'\subset J'$ implies $I\subset J$ for all
ideals $J$ of $\Cst(\mcA)$,
\item $\;I'\subset P'$ implies $I\subset P$ for all
$P\in\Prim\Cst(\mcA)$, i.e., $h(I)=\Psi^{-1}(h(I'))$,
\item $\;I'$ is dense in $I$ \wrt\ the $\Cst$-norm,
\item $\;\Cst(\mcA/I')\cong\Cst(\mcA)/I$.
\end{enumerateroman}
\end{defn}

In the introduction of \cite{Boid2} Boidol defined
$\ast$-regularity of Banach $\ast$-algebras. We
restate his definition and add the notion of
primitive $\ast$-regularity.

\begin{defn}\label{Bsa_def:star_regular}
A Banach $\ast\,$-algebra $\mcA$ is called (primitive)
$\ast\,$-regular if and only if every closed (primitive)
ideal of $\Cst(\mcA)$ is $\mcA$-determined.
\end{defn}

The group algebra $\LG$ of a locally compact group $G$ is
a $\ast$-semisimple Banach $\ast\,$-algebra with bounded
approximate identities. We say that $G$ is (primitive)
$\ast$-regular if $\LG$ has this property. Similarly
$\ast$-regularity of (real) Lie algebras $\frakg$ is
defined by means of the (unique) connected, simply
connected Lie group $G$ with $\Lie(G)=\frakg$.\\\\
Part \textit{(ii)} of the next lemma shows that
Definition~\ref{Bsa_def:star_regular} is equivalent
to Boidol's original definition, a characterization
which has already been proved in~\cite{Boid1}.

\begin{lem}\label{Bsa_lem:characterize_regularity}\hfill
\begin{enumerateroman}
\item If $\mcA$ is primitive $\ast\,$-regular, then
$\Psi:\Prim\Cst(\mcA)\to\Prim_{\ast}\mcA$ is injective.
\item A Banach $\ast\,$-algebra $\mcA$ is $\ast\,$-regular
if and only if $\Psi$ is a homeomorphism with respect to the
Jacobson topologies on $\Prim\Cst(\mcA)$ and
$\Prim_{\ast}\mcA$.
\end{enumerateroman}
\end{lem}

\begin{proof}
If $\mcA$ is primitive $\ast\,$-regular, then
$P=\overline{\Psi(P)}$ is uniquely determined by
$\Psi(P)$ for all $P\in\Prim\Cst(\mcA)$. This
proves \textit{(i)}. In order to prove \textit{(ii)},
let us suppose that $\mcA$ is $\ast\,$-regular. Since
$\Psi$ is a continuous bijection, it suffices to prove
that $\Psi$ maps closed sets onto closed sets. But if
$X$ is a closed subset of $\Prim\Cst(\mcA)$, then there
exists a closed ideal $I$ of $\Cst(\mcA)$ such that
$X=h(I)$ and we see that $\Psi(X)=h(I')$ is closed in
$\Prim_{\ast}\mcA$ because $I$ is $\mcA$-determined.
Now we prove the opposite implication. Assume that
$\Psi$ is a homeomorphism, $I$ a closed ideal of
$\Cst(\mcA)$, and $P\in\Prim\Cst(\mcA)$ such that
$I'\subset P'$. Define $X=h(I)$. Since $I'=k(\Psi(X))$,
it follows
$$h(I')=h(k(\Psi(X)))=\overline{\Psi(X)}=\Psi(X)$$
because $\Psi$ maps closed sets onto closed sets. Now
$P'\in\Psi(X)$ implies $P\in X$ so that $P\supset I$
because $\Psi$ is injective. This proves the asserted
equivalence.
\end{proof}

Because of its technical importance we state the
following fact as a lemma, but we omit the easy proof.

\begin{lem}\label{Bsa_lem:quotients_of_algebras}
Let $I\subset J$ be closed ideals of $\Cst(\mcA)$
such that $I$ is $\mcA$-determined. Then $J$ is
$\mcA$-determined if and only if the ideal $J/I$ of
$\Cst(\mcA)/I=\Cst(\mcA/I')$ is $\mcA/I'$-determined.
\end{lem}

This lemma can be applied in the following situation:
If $A$ is a closed normal subgroup of $G$ and
$\dot{G}=G/A$, then $T_Af\,(\dot x)=\int_A\,f(xa)\;da$
defines a quotient map of Banach $\ast\,$-algebras from
$\LG$ onto $L^1(\dot{G})$ which extends to a quotient
map from $\CG$ onto $\Cst(\dot{G})$, compare p.~68
of~\cite{Rei}. It is easy to see that $I=\ker_{\CG}T_A$
is $\LG$-determined.

\begin{lem}\label{Bsa_lem:finite_intersections}
A finite intersection of $\mcA$-determined ideals is
$\mcA$-determined.
\end{lem}

\begin{proof}
Let $I_1$ and $I_2$ be $\mcA$-determined ideals of
$\Cst(\mcA)$. Let $P\in\Prim\Cst(\mcA)$ such that
$I_1'\mdot I_2'\subset I_1'\cap I_2'\subset P'$.
Since $P'$ is a prime ideal of $\mcA$, it follows
$I_1'\subset P'$ or $I_2'\subset P'$. Since $I_1$
and $I_2$ are $\mcA$-determined, we obtain
$I_1\subset P$ or $I_2\subset P$ and thus
$I_1\cap I_2\subset P$. Consequently $I_1\cap I_2$
is $\mcA$-determined and the assertion of this
lemma follows by induction.
\end{proof}

\begin{rem}
Here are a few examples of $\ast\,$-regular Banach
$\ast$-algebras: If $G$ is a connected locally compact
group such that its Haar measure has polynomial growth,
then $G$ is $\ast\,$-regular. Boidol proved this fact
in Theorem~2 of~\cite{Boid1} based on ideas of Dixmier
in~\cite{Dix2}. Jenkins has shown in Theorem~1.4
of~\cite{Jen} that connected nilpotent Lie groups have
polynomial growth. If $G$ is a metabelian connected
locally compact group, then $G$ is $\ast\,$-regular,
see Theorem~3.5 of~\cite{Boid2}. Moreover the following
is true: If $G$ is a compactly generated, locally
compact group with polynomial growth and if $w$ is a
symmetric weight function on $G$ which satisfies
the non-abelian-Beurling-Domar condition~(BDna)
of~\cite{Dziu_et_alii}, then $L^1(G,w)$ is
$\ast$-regular. Compare Proposition~5.2 and
Theorem~5.8 of~\cite{Dziu_et_alii}.
\end{rem}

In the next paragraphs we formulate sufficient criteria
for ideals of the group algebra $\CG$ of exponential
Lie groups to be $\LG$-determined, see Proposition~%
\ref{eLg_prop:induced_from_regular_subgroup} and
Proposition~\ref{eLg_prop:closed_orbit_is_sufficient}.

\section{Inducing primitive ideals from
a stabilizer}\label{eLg_sec:inducing_ideals}
We shall use the concept of the adjoint algebra
(double centralizer algebra) of a Banach $\ast$-algebra,
compare Paragraph~3 of~\cite{Lep1} and Chapter~2.3
of~\cite{RaeWil}. Let $\Cc{G}$ denote the continuous
functions of compact support on $G$. If $H$ is a
closed subgroup of $G$, then $\Cc{H}$ acts as an
algebra of double centralizers
on $\Cc{G}$ by convolution
\[(a\ast f)\;(x)=\int_H\;a(h)\;f(h^{-1}x)\;dh\]
from the left, and by
\[(f\ast a)\;(x)=\int_H\;f(xh)\;\Delta_{G,H}(h^{-1})\;%
a(h^{-1})\;dh\]
from the right where $\Delta_{G,H}(h)=\Delta_H(h)%
\Delta_G(h)^{-1}$. These actions extend to
actions of $\CH$ on $\CG$ such that
$(a\,\ast\,f)^{\ast}=f^{\ast}\,\ast\,a^{\ast}$ and
$f\,\ast\,(a\,\ast\,g)=(f\,\ast\,a)\,\ast\,g$
for all $f,g\in\CG$ and $a\in\CH$.\\\\

\begin{defn}\label{eLg_def:induced_ideals}
Let $H$ be a closed subgroup of the locally compact
group $G$. If $J$ is an ideal of $\CH$, then
\[\ind_H^G(J)=(\,\CG\,\ast\,J\ast\CG\,)\clos\]
denotes the induced ideal of $\CG$. If $I$ is an
ideal of $\CG$, then the ideal
\[\res_H^G(I)=\{a\in\CH:a\ast\CG\subset I\}\]
is its restriction to $H$. An ideal $I$ of $\CG$
is said to be induced from $H$ if there exists
an ideal $J$ of $\CH$ such that $I=\ind_H^G(J)$.
\end{defn}

If $I=\ker_{\CG}\pi$ for some unitary representation
$\pi$ of $G$, then $\res_H^G(I)=\ker_{\CH}\pi$. If
$I$ is induced from $H$, then $I=\ind_H^G(\res_H^G(I))$.
Note that $I=\ind_H^G(J)$ is minimal among all ideals
of $\CG$ whose restriction contains $J$.\\

It is interesting to compare our definition of induced
ideals to that of Green and Rieffel in Section~3
of~\cite{Green1} involving $\Cst$-imprimitivity
bimodules. To this end we assume that there exists a
$G$-invariant measure on the homogeneous space $G/H$
so that the character $\Delta_{G,H}$ of $H$ is trivial.
This is the case e.g.\ if $H$ is a normal subgroup of $G$.
We follow the considerations of Section~4 of Rieffel's
article~\cite{Rief1}. Note that the right action of
$\Cc{H}$ on $X_0=\Cc{G}$ defined in \cite{Rief1} coincides
with that of convolution from the right because the
function $\gamma=\Delta_{G,H}^{-1/2}$ used there is also
trivial. The $\Cc{H}$-valued inner product
\[\langle\,f\,|\,g\,\rangle_{\Cc{H}}\;(h)\;=\;
(f^{\ast}\ast g)(h)\;=\;
\int_G\;\overline{f(y)}\,g(yh)\,dy\]
defines a norm $|\,f\,|_{\CH}=|\langle\,f\,|\,f\,%
\rangle_{\Cc{H}}|^{1/2}$ on $X_0$ where the norm on the
right is the $\Cst$-norm of $\CH$.  Further $\Cc{G}$
acts on $X_0$ by convolution from the left so that $X_0$
becomes a $\Cc{G}\,$-$\,\Cc{H}$-bimodule, and
$\langle\,f\,|\,g\,\rangle_{\Cc{G}}=f\ast g^{\ast}$
defines a $\Cc{G}$-valued inner product
$_{\CG}\langle\;\mdot\;|\;\mdot\;\rangle$ on $X_0$.
Completion of $X_0$ with respect to the norm
$|\;\mdot\;|_{\CH}$ gives a right-$\CH$-rigged space
$X$ on which $\CG$ acts from the left. From
$$\ind_H^G(J)=\overline{\lspan}\,\{\,_{\CG}\langle\,%
f\ast a\,|\,g\,\rangle\,:f,g\in X\text{ and }%
a\in J\,\}=X-\ind_H^G(J)$$
we learn that, at least in the case of $\Delta_{G,H}$ being
trivial, our definition coincides with that of Rieffel
and Green.\\\\
It is well-known that for $\Cst$-\textbf{imprimitivity}
bimodules $X$ the Rieffel correspondence $X-\ind_H^G$ is
compatible with inducing representations in
the sense that
\begin{equation}\label{eLg_equ:induced_ideal}
X-\ind_H^G(\ker\sigma)=\ker(X-\ind_H^G\,\sigma),
\end{equation}
compare Chapter~3.3 of~\cite{RaeWil}. But in general the
bimodule $X$ from above is \textbf{not} a
$\CG\,$-$\,\CH$-imprimitivity bimodule because the
crucial equality $_{\CG}\langle\,f\,|\,g\,\rangle\ast h=
f\ast\langle\,g\,|\,h\,\rangle_{\CH}$ is not necessarily
satisfied. The norms $|\;\mdot\;|_{\CH}$ and
$|\;\mdot\;|_{\CG}$ might be different. In fact, the
imprimitivity algebra of the $\CH$-rigged space $X$
is known to be isomorphic to the covariance algebra
$\Cst(G,\mcC_{\infty}(G/H))$. As we will see,
Equation~\ref{eLg_equ:induced_ideal} holds true for
the $\Cst$-bimodule $X$ defined above if $G/H$ is
amenable.\\\\
In analogy to results of Leptin~\cite{Lep2}
and Hauenschild, Ludwig~\cite{HauLud} for the
$L^1$-case, we will characterize those ideals $I$
of $\CG$ which are induced from a given closed normal
subgroup $H$ of $G$. This turns out to be possible
if $H$ is normal and $G/H$ amenable. In order to prepare
the proof of Theorem~\ref{eLg_thm:induced_ideals} we recall
the well-known restriction-induction-lemma of Fell,
see Theorem~3.1 and Lemma~4.2 of~\cite{Fell1}.
A proof can also be found on~p.~32 of~\cite{LepLud}.
We presume the definition of induced
representations.

\begin{lem}\label{eLg_lem:restriction_induction}
Let $H$ be a closed subgroup of a locally compact
group $G$. Let $\pi$ be a unitary representation
of $G$ and $\pi\,|\,H$ its restriction to $H$.
\begin{enumerateroman}
\item If $\tau$ is a unitary representation
of $H$, then the Kronecker product
$\ind_H^G(\,(\pi\,|\,H)\otimes\tau)$
is unitarily equivalent to
$\pi\otimes\ind_H^G\tau$.
\item In particular $\ind_H^G(\pi\,|\,H)$ is
unitarily equivalent to $\pi\otimes\lambda$
where $\lambda$ denotes the left regular
representation of $G$ in $L^2(G/H)$.
\end{enumerateroman}
\end{lem}

Note that conjugation
$f^z(x)=\Delta_G(z^{-1})\,f(zxz^{-1})$ for
$f\in\LG$ and $z\in G$ extends to a strongly
continuous action of $G$ on $\CG$ by isometric
automorphisms. Using an approximate identity
of $\CG$, one can prove that every closed ideal $I$
of $\CG$ is two-sided translation-invariant, and
hence invariant under conjugation, i.e.~$I^z=I$.\\\\
If $H$ is a closed normal subgroup of $G$, then
$G$ acts on $H$ by conjugation $n^z=z^{-1}nz$.
Further $a^z(n)=\delta(z^{-1})\,a(n^{z^{-1}})$
for $a\in\LH$ and $z\in G$ yields a strongly
continuous, isometric action of $G$ on $\CH$.
If $I$ is a closed ideal of $\CG$, then
$J=\res_H^G(I)$ is a $G$-invariant ideal
of $\CH$, i.e.~$J^z=J$, because
$(a\ast f)^z=a^z\ast f^z$ and $I^z=I$.

\begin{lem}
\label{eLg_lem:restriction_to_normal_subgroups}
Let $H$ be a closed normal subgroup of a locally
compact group $G$. Let $\sigma$ be a unitary
representation of $H$ and $\pi=\ind_H^G\sigma$.
Then $\pi\,|\,H$ is weakly equivalent to
the orbit $G\mdot\sigma$ which means
\begin{equation*}
\ker_{\CH}\,\pi=k(G\mdot \sigma)=\bigcap\limits_{x\in G}\,%
\ker_{\CH}\,x\mdot\sigma\;.
\end{equation*}
\end{lem}
\begin{proof}
Let $\frakH$ be the representation space of $\sigma$.
As usual $\mcC_0^\sigma(G,\frakH)$ denotes the vector
space of all continuous functions on $G$ which satisfy
$\varphi(xh)=\sigma(h)^\ast\,\varphi(x)$ for $h\in H$,
$x\in G$ and have compact support modulo $H$. Then
$\pi=\ind_H^G\sigma$ is defined in $L^2_\sigma(G,\frakH)$,
the completion of $\mcC_0^\sigma(G,\frakH)$ with respect
to the $L^2$-norm given by integration with respect to
the Haar measure of the group $G/H$. We get
\[\pi(h)\varphi\;(x)=\varphi(h^{-1}x)=
\sigma(h^x)\mdot\varphi(x)\]
for $h\in H$. It follows that $\pi\,|\,H$ is given by
$\pi(a)\varphi\;(x)=\sigma(a^x)\mdot\varphi(x)$ for
$a\in\CH$. Hence $\pi$ is essentially a direct integral
of the representations $\{x\mdot\sigma:x\in G\}$ so
that the assertion of this lemma becomes clear.
\end{proof}

The importance of the left regular representation
$\lambda$ of~$G$ in $L^2(G/H)$ has already been
indicated by Lemma~\ref{eLg_lem:restriction_induction}.

\begin{defn}\label{eLg_defn:invariant_ideals}
Let $H$ be a closed normal subgroup of a locally
compact group $G$. An ideal $I$ of $\CG$ is said
to be $(G/H)\widehat{\;\;}$-invariant if $\pi$ is
weakly equivalent to $\pi\otimes\lambda$ (in
symbols $\pi\approx\pi\otimes\lambda$) for all
unitary representations $\pi$ of $G$ such that
$I=\ker_{\CG}\pi$.
\end{defn}

Theorem~1 of~\cite{Fell2} shows that
$\pi\approx\pi\otimes\lambda$ for at least one
such $\pi$ is sufficient for $I$ to be
$(G/H)\widehat{\;\;}$-invariant. Now we can
state the announced characterization of
induced ideals.

\begin{thm}\label{eLg_thm:induced_ideals}
Let $H$ be a closed normal subgroup of a locally
compact group $G$ such that $G/H$ is amenable. Then
there are equivalent:
\begin{enumerateroman}
\item $I=\ind_H^G(\res_H^G(I))$ is induced from $H$.
\item $I=\ker_{\CG}\pi$ is the kernel of some induced
representation $\pi=\ind_H^G\sigma$.
\item $I$ is $(G/H)\widehat{\;\;}$-invariant.
\end{enumerateroman}
\end{thm}
\begin{proof}
First we verify \textit{(i)}~$\Rightarrow$~\textit{(ii)}.
Suppose that $I$ is induced from $H$. Since $J=\res_H^G(I)$
is a $G$-invariant ideal of $\CH$, its hull
$\Omega=h(J)\subset\widehat{H}$ is $G$-invariant, too.
Define $\sigma=\sum_{\tau\in\Omega}^\oplus\tau$ and
$\pi=\ind_H^G\sigma$. Lemma~\ref{eLg_lem:%
restriction_to_normal_subgroups} implies
$\ker_{\CH}\pi=k(G\mdot\sigma)=k(\Omega)=J$. Hence
$I=\ind_H^G(J)\subset\ker_{\CG}\pi$. We must prove
the opposite inclusion: Let $\rho\in\widehat{G}$
be arbitrary such that $I\subset\ker_{\CG}\rho$. Then
$k(G\mdot\sigma)=J\subset\ker_{\CH}\rho$ which means
that $\rho\,|\,H$ is weakly contained in $G\mdot\sigma$
(in symbols $\rho\,|\,H\ll G\mdot\sigma$). Since $G/H$ is
amenable, we have $1_G\ll\lambda=\ind_H^G1_H$ and hence
$\rho\otimes 1_G\ll\rho\otimes\lambda$ by Theorem~1
of~\cite{Fell2}. For inducing representations is
continuous \wrt\ the Fell topologies of $\widehat{H}$
and $\widehat{G}$, it follows from part~\textit{(ii)} of
Lemma~\ref{eLg_lem:restriction_induction} that
\[\rho\cong\rho\otimes 1_G\ll\rho\otimes\lambda\cong%
\ind_H^G(\rho\,|\,H)\ll\ind_H^G(G\mdot\sigma)\approx%
\ind_H^G\sigma=\pi\]
because the representations $\ind_H^G(z\mdot\sigma)$,
$z\in G$, are all unitarily equivalent. Thus
$\ker_{\CG}\pi\subset\ker_{\CG}\rho$. Since $I$ is
the intersection of all primitive ideals of $\CG$
containing $I$ by Theorem~2.9.7 of~\cite{Dix3}, we
obtain $I=\ker_{\CG}\pi$.\\\\
Next we show \textit{(ii)}$~\Rightarrow$~\textit{(iii)}.
Suppose that $I=\ker_{\CG}\pi$ for some
$\pi=\ind_H^G\sigma$. By Lemma~\ref{eLg_lem:%
restriction_to_normal_subgroups} we know
$\pi\,|\,H\approx G\mdot\sigma$. Thus
$\pi\otimes\lambda\cong\ind_H^G(\pi\,|\,H)
\approx\ind_H^G\sigma=\pi$ which proves $I$ to be
$(G/H)\widehat{\;\;}$-invariant.\\\\
Finally we prove \textit{(iii)}~$\Rightarrow$~\textit{(i)}.
Suppose that $I$ is $(G/H)\widehat{\;\;}$-invariant.
Set $J=\res_H^G(I)$. Clearly $\ind_H^G(J)\subset I$.
It remains to verify the opposite inclusion: Choose
a unitary representation $\pi$ of $G$ such that
$I=\ker_{\CG}\pi$. Let $\rho\in\widehat{G}$ be
arbitrary such that $\ind_H^G(J)\subset\ker_{\CG}\rho$.
Then $\ker_{\CH}\pi=J\subset\ker_{\CH}\rho$ and hence
$\rho\,|\,H\ll\pi\,|\,H$. Since $G/H$ is amenable,
we get
\[\rho=\rho\otimes 1\ll\rho\otimes\lambda
=\ind_H^G(\rho\,|\,H)\ll\ind_H^G(\pi\,|\,H)
=\pi\otimes\lambda\approx\pi\]
because $I$ is $(G/H)\widehat{\;\;}$-invariant. Thus
$I=\ker_{\CG}\pi\subset\ker_{\CG}\rho$. Now
Theorem~2.9.7 of~\cite{Dix3} implies $I=\ind_H^G(J)$.
The proof is complete.
\end{proof}

The proof of \textit{(i)}~$\Rightarrow$\textit{(ii)}
of Theorem~\ref{eLg_thm:induced_ideals} shows that
$J=\res_H^G(\ind_H^G(J))$ for every $G$-invariant
ideal $J$ of $\CH$. The preceding results
can be summarized as follows:

\begin{thm}\label{eLg_thm:induction_and_%
restriction_are_bijections}
Let $H$ be a closed normal subgroup of a locally
compact group $G$ such that $G/H$ is amenable.
Induction and restriction give bijections
between the set of all $(G/H)\widehat{\;\;}$-invariant
ideals $I$ of $\CG$ and the set of all $G$-invariant
ideals $J$ of $\CH$ which are inverses of one
another.
\end{thm}

An immediate consequence is

\begin{cor}
\label{eLg_cor:intersection_of_induced_ideals}
Let $H$ be a closed normal subgroup of a locally
compact group $G$ such that $G/H$ is amenable. If
the ideals $\{I_k:k\in\Lambda\}$ of $\CG$ are
induced from $H$, then their intersection
$I=\bigcap\,\{I_k:k\in\Lambda\}$ is also
induced from $H$.
\end{cor}
\begin{proof}
Let $\pi_k$ be a unitary representation of $G$
such that $I_k=\ker_{\CG}\pi_k$. We know
$\pi_k\otimes\lambda\approx\pi_k$ by
Theorem~\ref{eLg_thm:induced_ideals}. If we
define $\pi=\sum_{k\in\Lambda}^\oplus\pi_k$, then
$I=\ker_{\CG}\pi$ and $\pi\otimes\lambda\approx%
\{\pi_k\otimes\lambda:k\in\Lambda\}\approx%
\{\pi_k:k\in\Lambda\}\approx\pi$. Thus $I$ induced
from $H$ by Theorem~\ref{eLg_thm:induced_ideals}.
\end{proof}

Suppose that $H$ is a coabelian normal subgroup
of $G$ so that $G/H$ is amenable as an abelian
group. In this case $(\chi\mdot f)(x)=\chi(x)f(x)$
for $f\in\LG$ extends to an isometric, strongly
continuous action of the Pontryagin dual
$(G/H)\widehat{\;\;}$ on $\CG$. Note that
$\pi(\chi\mdot f)=(\pi\otimes\chi)(f)$ for
any unitary representation $\pi$ of $G$.

\begin{cor}\label{eLg_cor:ideals_induced_from_%
coabelian_subgroups}
Let $H$ be a coabelian normal subgroup of $G$. An ideal
$I$ of $\CG$ is induced from $H$ if and only if it is
$(G/H)\widehat{\;\;}$-invariant in the sense that
$\chi\mdot I=I$ for all $\chi\in(G/H)\widehat{\;\;}$.
\end{cor}
\begin{proof}
Let $\pi$ be a unitary representation of $G$ such that
$I=\ker_{\CG}\pi$. Theorem~\ref{eLg_thm:induced_ideals}
shows that $I$ is induced from $H$ if and only if
$\pi\approx\pi\otimes\lambda$. Since $G/H$ is
abelian, it follows
$\lambda\approx\{\chi:\chi\in(G/H)\widehat{\;\;}\}$
and hence
\[\ker_{\CG}\pi\otimes\lambda=
\bigcap_{\chi\in (G/H)\widehat{\;\;}}
\ker_{\CG}\pi\otimes\chi\subset\ker_{\CG}\pi\;.\]
Thus we see that $\pi\otimes\lambda\approx\pi$ if
and only if $\ker_{\CG}\pi\otimes\chi=\ker_{\CG}\pi$
for all $\chi$. This is the case if and
only if $\chi\mdot I=I$ for all
$\chi\in(G/H)\widehat{\;\;}$.
\end{proof}

The preceding corollary displays a close connection
to the $L^1$-results of Leptin and Hauenschild, Ludwig.
In 1968 Leptin characterized the induced ideals of
generalized $L^1$-algebras / twisted covariance
algebras $L^1(G,\mcA,\tau)$, see  Satz~8 and
Satz~9 of~\cite{Lep2}. His results imply
Theorem~\ref{eLg_thm:induced_ideals_in_L1} and
Lemma~\ref{eLg_lem:convolution_with_Ginvariant_ideals}
below. In 1981 Hauenschild and Ludwig gave a different
proof of Theorem~\ref{eLg_thm:induced_ideals_in_L1}
using $L^1$-$L^\infty$-duality, see Theorem~2.3
of~\cite{HauLud}. These $L^1$-results hold true
without the additional assumption of amenability.\\\\
An ideal $I$ of $\LG$ is said to be induced from $H$
if there exists an ideal $J$ of $\LH$ such that
$I=\ind_H^G(J)=\left(\,\LG\ast J\ast\LG\,\right)\clos$.

\begin{thm}
\label{eLg_thm:induced_ideals_in_L1}
Let $H$ be a closed normal subgroup of a locally compact
group $G$.  An ideal $I$ of $\LG$ is induced from $H$
if and only if it is $\mcC_\infty(G/H)$-invariant. Induction
and restriction gives a bijection between the set
of all $\mcC_\infty(G/H)$-invariant ideals $I$ of $\LG$
and the set of all $G$-invariant ideals $J$ of $\LH$.
\end{thm}

Here $\mcC_\infty(G/H)$ denotes the continuous functions
on $G/H$ vanishing at infinity.

\begin{lem}
\label{eLg_lem:convolution_with_Ginvariant_ideals}
Let $H$ be a closed normal subgroup of a locally
compact group $G$. If $J$ is a closed, $G$-invariant
ideal of $\LH$, then $J\ast\LG$ is contained in the
closure of $\LG\ast J$. Similarly $\LG\ast J$ is
contained in the closure of $J\ast\LG$.
\end{lem}

This implies $I=\ind_H^G(J)=(J\ast\LG)\clos=(\LG\ast J)\clos$.
For $G$-invariant $\Cst$-ideals we even know
$J\ast\CG=\CG\ast J$ by Corollary~2.3 of the main lemma
in~\cite{Rief2}.\\\\
Now we can state our first criterion for ideals of $\CG$
to be $\LG$-determined.

\begin{prop}\label{eLg_prop:induced_from_%
regular_subgroup}
Let $G$ be a locally compact group and $H$ a
$\ast$-regular closed subgroup. If the ideal $I$ of
$\CG$ is induced from $H$, then $I$ is $\LG$-determined.
\end{prop}

\begin{proof}
Let $J=\res_H^G(I)$ so that $I=\ind_H^G(J)$. If
$\rho\in\widehat{G}$ with $I'=I\cap\LG\subset\ker_{\LG}\rho$,
then $J'\subset\ker_{\LH}\rho$. Since $H$ is $\ast$-regular,
it follows $J\subset\ker_{\CH}\rho$. This implies
$I=\ind_H^G(J)\subset\ind_H^G(\ker_{\CH}\rho)%
\subset\ker_{\CG}\rho$.
\end{proof}

In the rest of this article we will focus on exponential
Lie groups (i.e.~connected, simply connected, solvable
Lie groups $G$ such that the exponential map
$\exp:\frakg\to G$ is a global diffeomorphism). We
will use the construction of irreducible representations
$\pi=\mcK(f)=\ind_P^G\chi_f$ via Pukanszky / Vergne
polarizations $\frakp$ at $f$ and the bijectivity of
the Kirillov map $\mcK:\frakg^\ast/\coAd(G)\to\widehat{G}$,
see Chapters~4 and~6 of~\cite{Bern_et_alii}, and Chapter~1
of~\cite{LepLud}. Mostly we regard $\mcK$ as a map from
$\frakg^\ast$ onto $\widehat{G}$ which is constant
on coadjoint orbits. 

\begin{lem}
Let $G$ be an exponential Lie group with Lie
algebra $\frakg$. Let $f\in\frakg^\ast$ and
$q\in[\frakg,\frakg]^\bot\subset\frakg^\ast$. If
we define $\pi=\mcK(f)$ and the character
$\alpha(\exp X)=e^{iq(X)}$ of $G$, then
$\mcK(f+q)$ and $\pi\otimes\alpha$ are
unitarily equivalent.
\end{lem}
\begin{proof}
Let $\frakp\subset\frakg$ be a Pukanszky polarization
at $f$, and hence also at $f+q$. Let $\chi_f$ and
$\chi_{f+q}$ denote characters of~$P$ with differential
$f$ and $f+q$. By definition of the Kirillov map we have
$\pi=\ind_P^G\chi_f$ and $\rho=\mcK(f+q)=\ind_P^G%
\chi_{f+q}$. Now one verifies easily that
$(U\varphi)\,(x)=\overline{\alpha(x)}\;\varphi(x)$
defines a unitary isomorphism from $\frakH_{\pi}=%
L^2_{\chi_f}(G)$ onto $\frakH_\rho=L^2_{\chi_{f+q}}(G)$
such that $\rho=U(\pi\otimes\alpha)U^{-1}$. This
proves our claim.
\end{proof}

The next proposition enlightens the significance of
the 'stabilizer' $M$.

\begin{prop}\label{eLg_prop:ideal_induced_from_M}
Let $G$ be an exponential Lie group, $\frakn$ a coabelian
(nilpotent) ideal of its Lie algebra $\frakg$, and
$f\in\frakg^\ast$. Let $M$ denote the connected subgroup
of $G$ with Lie algebra $\frakm=\frakg_f+\frakn$. If
$\pi=\mcK(f)$, then the primitive ideal $\ker_{\CG}\pi$
is induced from the stabilizer $M$.
\end{prop}
\begin{proof}
First we observe that the orbit $\coAd(G)f$ is saturated
over $\frakm$: Let $G_l$ denote the (connected)
stabilizer of $l=f\,|\,\frakn$ in $G$. Since
$\coAd(G_l)f=f+\frakm^\bot$, it follows
$\coAd(G)f=\coAd(G)f+\frakm^\bot$, compare p.~23
of~\cite{Bern_et_alii}. Now the preceding lemma
implies $\pi\otimes\alpha=\mcK(f+q)=\mcK(f)=\pi$
for all $q\in\frakm^\bot$ and characters
$\alpha(\exp X)=e^{iq(X)}$ of $G/M$ proving
$\ker_{\CG}\pi$ to be $(G/M)\widehat{\;\;}$-invariant.
Hence $\ker_{\CG}\pi$ is induced from $M$ by Corollary~%
\ref{eLg_cor:ideals_induced_from_coabelian_subgroups}.
\end{proof}

\boldmath
\section{The ideal theory of $\ast\,$-regular
 exponential Lie groups}
\unboldmath\label{eLg_sec:ideal_theory_of_exponential_groups}
The results of this subsection are not new. They can be
found in Boidol's paper~\cite{Boid2}, and in a more
general context in~\cite{Boid3}. For the convenience
of the reader we give a short proof for the if-part
of Theorem~5.4 of~\cite{Boid2} using the results of
the previous section. The following definition has
been adapted from the introduction of~\cite{Boid3}.

\begin{defn}
Let $G$ be a locally compact group. If $A$ is a closed
normal subgroup of $G$ and $\dot{G}=G/A$, then $T_A$
denotes the quotient map from $\CG$ onto
$C^\ast(\dot{G})$. We say that a closed ideal $I$ of
$\CG$ is essentially induced from a $\ast$-regular
subgroup if there exist closed subgroups $A\subset H$
of $G$ with $A$ normal in $G$ such that the following
conditions are satisfied:
\begin{enumerateroman}
\item $\;\ker_{\CG}\,T_A\subset I$,
\item $\;H/A$ is $\ast$-regular,
\item $\;I$ is induced from $H$ in the sense of
Definition~\ref{eLg_def:induced_ideals}.
\end{enumerateroman}
\end{defn}

Recall that connected locally compact groups whose
Haar measure have polynomial growth are $\ast\,$-regular,
and that connected nilpotent Lie groups have polynomial
growth. If we pass to the quotient $\dot{G}$ by
Proposition~\ref{Bsa_lem:quotients_of_algebras}, then it
follows from Proposition~\ref{eLg_prop:induced_from_%
regular_subgroup} that all ideals $I$ of $\CG$ which are
essentially induced from a $\ast\,$-regular subgroup
are $\LG$-determined.

\begin{defn}
Let $\frakg$ be an exponential Lie algebra and
$\frakn=[\frakg,\frakg]$ its commutator ideal. We say
that $\frakg$ satisfies condition (R) if the following
is true: If $f\in\frakg^\ast$ is arbitrary and
$\frakm=\frakg_f+\frakn$ is its stabilizer, then $f=0$
on $\frakm^{\infty}=\bigcap_{k=1}^\infty C^k\frakm$.
Here the $C^k\frakm$ are the ideals of the descending
central series. Recall that $\frakm^{\infty}$ is the
smallest ideal of $\frakm$ such that
$\frakm/\frakm^{\infty}$ is nilpotent.
\end{defn}

Note that the stabilizer $\frakm=\frakg_f+\frakn$
depends only on the orbit $\coAd(G)f$. The following
observation is extremely useful: Let $f\in\frakg^\ast$
and $\frakm=\frakg_f+\frakn$ be its stabilizer such
that $\frakm/\frakm^\infty$ is nilpotent. If
$\gamma_1,\ldots,\gamma_r$ are the roots
of $\frakg$, then we define the ideal
$\tilde{\frakm}=\bigcap_{i\in S}\ker\gamma_i$
of $\frakg$ where $S=\{i:\ker\gamma_i\supset\frakm\}$.
It is easy to see that $\frakm\subset\tilde{\frakm}$
and that $\tilde{\frakm}/\frakm^\infty$ is nilpotent,
too. Further there are only finitely many ideals
$\tilde{\frakm}$ of this kind.

\begin{thm}\label{eLg_thm:conditionR_implies_regularity}
Let $G$ be an exponential Lie group such that its Lie
algebra $\frakg$ satisfies condition $(R)$. Then any
ideal $I$ of $\CG$ is a finite intersection of ideals
which are essentially induced from a nilpotent normal
subgroup. In particular $G$ is $\ast\,$-regular.
\end{thm}

\begin{proof}
Let $I\lhd\CG$ be arbitrary. Since $I=k(h(I))$ by
Theorem~2.9.7 of~\cite{Dix3}, there is a closed,
$\coAd(G)$-invariant subset $\Lambda$ of $\frakg^\ast$
such that $I=\bigcap\,\{\ker_{\CG}\mcK(f)\,:f\in\Lambda\}$.
Further there exists a decomposition
$\Lambda=\bigcup_{k=1}^r \Lambda_k$ and ideals
$\{\tilde{\frakm}_k:1\le k\le r\}$ of $\frakg$
as in the preceding remark such that
$\frakg_f+\frakn\subset\tilde{\frakm}_k$ for all
 $f\in\Lambda_k$ where $\frakn=[\frakg,\frakg]$.
By induction in stages it follows from
Proposition~\ref{eLg_prop:ideal_induced_from_M}
that $\ker_{\CG}\mcK(f)$ is induced from $\tilde{M}_k$
for all $f\in \Lambda_k$. Let us define
$I_k=\bigcap\,\{\ker_{\CG}\mcK(f)\,:f\in\Lambda_k\}$.
Since $f=0$ on $\tilde{\frakm}_k^\infty$ by
condition~(R) and $\tilde{M}_k/\tilde{M}_k^\infty$
is nilpotent, we conclude from
Corollary~\ref{eLg_cor:intersection_of_induced_ideals}
that $I_k$ is essentially induced from a nilpotent
(and hence $\ast\,$-regular) normal subgroup. Finally
Lemma~\ref{Bsa_lem:finite_intersections} implies that
the ideal $I=\bigcap_{k=1}^r\,I_k$ is $\LG$-determined.
\end{proof}

\section{Closed orbits in the unitary dual  of the
nilradical}\label{eLg_sec:same_orbit_on_nilradical}
First we recall how to compute the $\Cst$-kernel of
$\pi\,|\,N$ in the Kirillov picture, compare Theorem~9
in Section~5 of Chapter~1 in~\cite{LepLud}. Note that
the linear projection $r:\frakg^\ast\onto\frakn^\ast$
given by restriction is $\coAd(G)$-equivariant so
that $r(\coAd(G)f)=\coAd(G)l$.

\begin{lem}\label{eLg_lem:restriction}
Let $G$ be an exponential Lie group and $\frakn$ a
coabelian ideal of its Lie algebra~$\frakg$. Let
$f\in\frakg^\ast$, $\pi=\mcK(f)\in\widehat{G}$,
$l=f\,|\,\frakn$, and $\sigma=\mcK(l)\in\widehat{N}$.
Then
\begin{equation}\label{eLg_equ:kernel_of_orbit}
\ker_{\CN}\pi=k(G\mdot\sigma)=\bigcap_{h\in\coAd(G)l}
\ker_{\CN}\mcK(h)\;.
\end{equation}
\end{lem}
\begin{proof}
The second equality is obvious because the Kirillov map
of $N$ is $G$-equivariant. By induction it suffices to
prove the first equality in the case $\dim\frakg/\frakn=1$.
First we assume $\frakg_f\subset\frakn$. Let us choose a
Pukanszky polarization $\frakp\subset\frakn$ at
$l\in\frakn^\ast$. It is easy to see that
$\frakp\subset\frakg$ is also a Pukanszky polarization
at $f\in\frakg^\ast$. By induction in stages we obtain
$\pi=\ind_P^G\chi_f=\ind_N^G\sigma$ so that
$\ker_{\CG}\pi=k(G\mdot\sigma)$ by Lemma~\ref{eLg_lem:%
restriction}. Next we assume $\frakg_f\not\subset\frakn$.
Using the concept of Vergne polarizations passing through
$\frakn$ we see that there exists a Pukanszky polarization
$\frakp\subset\frakg$ at $f\in\frakg^\ast$ such that
$\frakq=\frakp\cap\frakn$ is a Pukanszky polarization
at $l\in\frakn^\ast$. We point out that the restriction
of functions from~$G$ to~$N$ gives a linear isomorphism
$\mcC_0^{\chi_f}(G)\to\mcC_0^{\chi_l}(N)$ which
extends to a unitary isomorphism $U$ from
$\frakH_{\pi}=L^2_{\chi_f}(G)$ onto $\frakH_{\sigma}=%
L^2_{\chi_l}(N)$. Clearly $U$ intertwines $\pi\,|\,H$
and $\sigma$. On the other hand $\frakg=\frakg_f+\frakn$
implies $\coAd(G)l=\coAd(N)l$ and thus
$G\mdot\sigma=\{\sigma\}$. This proves
$\ker_{\CN}\pi=\ker_{\CN}\sigma=k(G\mdot\sigma)$.
\end{proof}

In the sequel we suppose that $\frakn$ is nilpotent
and coabelian. Note that the orbit $G\mdot\sigma%
\subset\widehat{N}$ is uniquely determined by
Equality~(\ref{eLg_equ:kernel_of_orbit}) because
it is locally closed (open in its closure):
Pukanszky showed in Corollary~1 of~\cite{Puk1}
that $\coAd(G)l$ is locally closed in~$\frakn^\ast$
and Brown proved in~\cite{Brown} that the Kirillov
map of the connected, simply connected, nilpotent
Lie group $N$ is a homeomorphism.\\\\
Our main result is Proposition~\ref{eLg_prop:%
closed_orbit_is_sufficient} which states that the
primitive ideal $\ker_{\CG}\pi$ is $\LG$-determined
if $G\,\mdot\,\sigma$ is closed in~$\widehat{N}$.
This result is a consequence of arguments closely
related to the classification of simple
$\LG$-modules, $G$ an exponential Lie group,
established by Poguntke in~\cite{Pog4}.\\\\
Let $\pi$, $f$, $\sigma$, $l$ be as in
Lemma~\ref{eLg_lem:restriction}. It is easy
to see that $\frakg=\frakg_l+\frakn$ is sufficient
for $G\mdot\sigma$ to be closed in $\widehat{N}$:
Theorem~3.1.4 of~\cite{CorGre} implies that
$\coAd(G)l=\coAd(N)l$ is closed in $\frakn^\ast$
because $N$ acts unipotently on~$\frakn^\ast$. Since
the Kirillov map of~$N$ is a homeomorphism, it follows
that $G\mdot\sigma=\{\sigma\}$ is closed in~$\widehat{N}$.
Alternatively one can resort to the results of Moore and
Rosenberg: It follows from Theorem~1 of~\cite{MooreRos}
that $\widehat{N}$ is a $T_1$-space so that its one-point
subsets are closed. Let us give a third proof of this
fact: Since $\LN$ is symmetric for nilpotent connected
Lie groups $N$ by Satz~2 of~\cite{Pog1}, it follows
$\Prim_{\ast}\LN=\Max\LN$ by~(6) of~\cite{Lep5} so
that points $\{\sigma\}$ are closed in $\widehat{N}$
because $N$ is $\ast$-regular.\\\\
Poguntke proved in~Theorem~7 of~\cite{Pog4} that if
$E$ is a simple $\LG$-module and $N$ is a connected,
coabelian, nilpotent subgroup of $G$, then there exists
a unique orbit $G\mdot\sigma\subset\widehat{N}$ such that
$\Ann_{\LN}(E)=k(G\mdot\sigma)$. More generally, Ludwig
and Molitor-Braun showed in \cite{LudMol} that if $T$
is a topologically irreducible, bounded representation
of $\LG$, then $\ker_{\LN}T=k(G\mdot\sigma)$ for some
$\sigma\in\widehat{N}$.\\\\
We need the following well-known facts about simple
modules and minimal hermitian idempotents. In the
following irreducible means topologically irreducible.

\begin{lem}\label{eLg_lem:associated_simple_modules}
Let $\mcB$ be Banach $\ast$-algebra and $\pi$ an
irreducible $\ast\,$-representation of $\mcB$
in a Hilbert space $\frakH$.
\begin{enumerateroman}
\item Let $\xi\in\frakH$ be non-zero. Then the
subspace $\pi(\mcB)\xi$ is non-zero and dense in
$\frakH$. If $I$ is an ideal of $\mcB$ such that
$I\not\subset\ker\pi$, then $\pi(I)\xi$ is also
non-zero and dense.
\item Suppose that the ideal $I$ of all $f\in\mcB$ such
that $\pi(f)$ has finite rank is non-zero. Then the
$\pi(\mcB)$-invariant subspace $E=\pi(I)\frakH$ generated
by $\{\pi(f)\eta:f\in I,\eta\in\frakH\}$ is a simple
$\mcB$-module such that $\Ann_{\mcB}(E)=\ker_{\mcB}\pi$.
\end{enumerateroman}
\end{lem}

\begin{proof}Part \textit{(i)} is obvious. The proof
of \textit{(ii)} follows Dixmier's proof of Th\'eor\`eme~2
in~\cite{Dix2}. Let $\xi\in E$ be non-zero. For every
$f\in I$ the subspace $\pi(f)\pi(\mcB)\xi$ is dense in
$\pi(f)\frakH$ so that $\pi(f)\pi(\mcB)\xi=\pi(f)\frakH$
because $\pi(f)\frakH$ is finite-dimensional. This proves
$\pi(f)\frakH\subset\pi(I)\xi$ for every $f\in I$. Thus
$E=\pi(I)\frakH=\pi(I)\xi$. The rest is obvious.
\end{proof}

A hermitian idempotent $q\in\mcB$ satisfies $q^2=q=q^\ast$.
We say that $q$ is minimal in~$\mcB$ if it is non-zero
and if $q\mcB q=\mC q$.

\begin{lem}\label{eLg_lem:algebras_with_minimal_%
hermitian_idempotents}
Let $\mcB$ be Banach $\ast$-algebra.
\begin{enumerateroman}
\item Let $\pi$ be a faithful irreducible
$\ast\,$-representation of $\mcB$ in a Hilbert space
$\frakH$. Then $q\in\mcB$ is a minimal hermitian
idempotent if and only if $\pi(q)$ is a one-dimensional
orthogonal projection.
\item Assume that there exist minimal hermitian
idempotents in $\mcB$. If $\pi,\rho$ are faithful
irreducible $\ast\,$-representations of $\mcB$, then
$\pi$ and $\rho$ are unitarily equivalent.
\end{enumerateroman}
\end{lem}
\begin{proof}
Clearly $q$ is a hermitian idempotent if and only if
$\pi(q)$ is an orthogonal projection because $\pi$ is
faithful. If $\pi(q)\frakH$ is a one-dimensional, then
$\pi(\mC q)=\mC\pi(q)=\pi(q)\,\pi(\mcB)\,\pi(q)=\pi(q\mcB q)$
and thus $q\mcB q=\mC q$ because $\pi$ is faithful. For
the converse assume $q\mcB q=\mC q$. Since $\pi(q\mcB q)$
and hence $\pi(q)$ acts irreducibly on $\pi(q)\frakH$, it
follows that this subspace is one-dimensional.\\\\
Now we prove \textit{(ii)}. Let $q\in\mcB$ be a minimal
hermitian idempotent and $\pi$,~$\rho$ faithful
irreducible $\ast\,$-representations in Hilbert spaces
$\frakH_{\,\pi}$ and $\frakH_{\,\rho}$. Since $\pi(q)$
and $\rho(q)$ are one-dimensional orthogonal projections
by \textit{(i)}, there exist unit vectors
$\xi\in\frakH_{\,\pi}$ and $\eta\in\frakH_{\,\rho}$ such
that  $\pi(q)=\langle\,-\,,\xi\rangle\,\xi$ and
$\rho(q)=\langle\,-\,,\eta\rangle\,\eta$. Let us
consider the positive linear functionals
$f_{\pi},f_{\rho}$ on $\mcB$ given by
$f_{\pi}(a)=\langle\pi(a)\xi,\xi\rangle$ and
$f_{\rho}(a)=\langle\rho(a)\eta,\eta\rangle.$
Since $q\mcB q$ is one-dimensional and
$f_{\pi}(q)=1=f_{\rho}(q)$, it follows
$f_{\pi}(a)=f_{\pi}(qaq)=f_{\rho}(qaq)=f_{\rho}(a)$
for all $a\in\mcB$, i.e., the positive linear forms of
the cyclic representations $\pi$ and $\rho$ coincide. Now
Proposition~2.4.1 of~\cite{Dix3} shows that $\pi$ and
$\rho$ are unitarily equivalent.
\end{proof}

Poguntke proved in~\cite{Pog3} that for exponential $G$
and $\pi\in\widehat{G}$ there exists some $q\in\LG$
such that $\pi(q)$ is a one-dimensional orthogonal
projection. Note that the canonical image of $q$
in $\LG/\ker_{\LG}\pi$ is a minimal hermitian
idempotent. Part~\textit{(ii)} of Lemma~\ref{eLg_lem:%
algebras_with_minimal_hermitian_idempotents} shows us
that $\ker_{\LG}\pi=\ker_{\LG}\rho$ for
$\pi,\rho\in\widehat{G}$ implies that $\pi$ and $\rho$
are unitarily equivalent. In particular $G$ is a
type I group. Furthermore the natural map
$\Psi:\Prim\Cst(G)\to\Prim_{\ast}\LG$ is injective,
which is necessary for $G$ to be primitive
$\ast\,$-regular by Lemma~\ref{Bsa_lem:%
characterize_regularity}.\\

If $E$ is a simple $\mcB$-module, then there exists a
complete norm on $E$ such that $|a\mdot\xi|\le|a|\,|\xi|$
for $a\in\mcB$ and $\xi\in E$: Recall that $E$ is
algebraically isomorphic to $\mcB/L$ for some maximal
modular left ideal $L$ which is closed in the Banach
algebra $\mcB$. The quotient norm of $E\cong\mcB/L$ has
the desired property. In particular we see that primitive
ideals $P=\Ann_{\mcB}(E)$ are closed. Furthermore primitive
ideals are prime. Hence the set $\Prim\mcB$ of all
primitive ideals of $\mcB$ can be endowed with the
Jacobson (hull-kernel) topology.\\

In the sequel we work with hermitian idempotents in the
adjoint algebra $\mcB^b$ of $\mcB$, compare~\cite{Lep1}, which
is also known as the multiplier or double centralizer
algebra of $\mcB$.
\ifthenelse{\boolean{details}}{ }
{For a proof of the following proposition we refer the
reader to Theorem~1 of~\cite{Pog4}.}

\begin{prop}\label{eLg_prop:corners}
Let $\mcB$ be a Banach $\ast\,$-algebra and
$q\in\mcB^b$ a hermitian idempotent.
\begin{enumerate}
\item $q\mcB q$ is a closed $\ast$-subalgebra of $\mcB$.
\item If $E$ is a simple $\mcB$-module, then there exists
a unique (simple) $\mcB^b$-module structure on $E$ such
that $M\mdot(a\mdot\xi)=(Ma)\mdot\xi$ for all
$M\in\mcB^b$, $a\in\mcB$, and $\xi\in E$.
\item If $E$ is a simple $\mcB$-module such that
$q\mdot E\neq 0$, then $q\mdot E$ is a simple
$q\mcB q$-module with annihilator
$\Ann_{q\mcB q}(q\mdot E)= q\mcB q\cap\Ann_{\mcB}(E)%
=q\Ann_{\mcB}(E)q$.
\item The assignment $[E]\mapsto [q\mdot E]$ gives a
bijection from the set of isomorphism classes of
simple $\mcB$-modules $E$ such that $q\mdot E\neq 0$
onto the set of isomorphism classes of simple
$q\mcB q$-modules.
\item Further $P\mapsto q\mcB q\cap P$ is a homeomorphism
from the open subset $\Prim\mcB\setminus h(\mcB q\mcB)$
onto $\Prim(q\mcB q)$ \wrt\ the Jacobson topology.
\end{enumerate}
\end{prop}

\ifthenelse{\boolean{details}}
{\begin{proof} A proof of parts \textit{1.} to \textit{4.}
of this proposition can also be found in~\cite{Pog4}.
\begin{enumerate}
\item[\textit{1.}] Clearly $q\mcB q$ is a
$\ast$-subalgebra of $\mcB$ because $q$ is hermitian.
The map $a\mapsto qaq$ is a continuous, linear
projection. Its image $q\mcB q$ is closed.
\item[\textit{2.}] Recall that $\mcB$ is an ideal
of $\mcB^b$. Let $E$ be a simple $\mcB$-module. If
$a\mdot\xi=0$, then $\mcB\mdot(Ma)\mdot\xi=(\mcB M)%
\mdot(a\mdot\xi)=0$ which implies $(Ma)\mdot\xi=0$.
Thus $M\mdot(a\mdot\xi)=(Ma)\mdot\xi$ defines a
$\mcB^b$-module structure on $E$.  The rest is obvious.
\item[\textit{3.}] Let $E$ be a simple $\mcB$-module.
Clearly $q\mdot E$ is a $q\mcB q$-module. If
$0\neq\xi\in q\mdot E$, then $(q\mcB q)\mdot\xi=%
q\mcB\mdot\xi=q\mdot E$. Thus $q\mdot E$ is simple.
The equality for its annihilator is clear.
\item[\textit{4.}] Since any simple $\mcB$-module
is isomorphic to one of the form $\mcB/L$, $L$ a
maximal left ideal of $\mcB$, the isomorphism
classes of simple $\mcB$-modules form a set.
Note that any $\mcB$-linear map is also
$\mcB^b$-linear.\\\\
The map $\alpha([E])=[q\mdot E]$ is well-defined because
any $\mcB$-linear isomorphism~$\varphi$ from $E_1$ onto
$E_2$ restricts to a $q\mcB q$-linear isomorphism
$\varphi'$ from $q\mdot E_1$ onto $q\mdot E_2$. Further
$\alpha$ is injective because any $q\mcB q$-linear
isomorphism $\varphi':q\mdot E_1\to q\mdot E_2$ extends
to a $\mcB$-linear isomorphism $\varphi:E_1\to E_2$: To
see this, choose a non-zero $\xi\in q\mdot E$ and
define $\varphi(a\mdot\xi)=a\mdot\varphi'(\xi)$.
Finally, it remains to verify that $\alpha$ is
surjective: Let $E'$ be a simple $q\mcB q$-module.
Since $q\mcB^b q\subset (q\mcB q)^b$, we can define
$E_0=\mcB\otimes_{q\mcB^b q}E'=\mcB q\otimes_{q\mcB^b q}E'$.
Observe that $q\mdot E_0=q\otimes_{q\mcB^b q}E'\cong E'$.
By Zorn's Lemma there exists a maximal $\mcB$-invariant
subspace $U$ of $E_0$ such that $U\cap q\mdot E_0=\{0\}$.
Put $E=E_0/U$. Clearly $q\mdot E\cong E'$. We claim that
$E$ is simple: If $\eta\not\in U$, then the
$\mcB$-invariant subspace $\tilde{U}=\mcB\mdot\eta+U$
satisfies $\tilde{U}\cap q\mdot E_0\neq\{0\}$ and hence
$q\mdot\tilde{U}\neq 0$. This implies
$q\mcB q\mdot\tilde{U}=q\otimes_{q\mcB^b q}E'$ and
$\mcB q\mcB q\mdot\tilde{U}=\mcB\otimes_{q\mcB^b q}E'=E_0$.
\item[\textit{5.}] Part~\textit{3.}\ implies
$\beta(P)=q\mcB q\cap P\in \Prim(q\mcB q)$ for all
$P\in\Prim(\mcB)\setminus h(\mcB q\mcB)$. It follows
from~\textit{4.}\ that any simple $q\mcB q$-module is
isomorphic to one of the form $q\mdot E$, $E$ a simple
$\mcB$-module. Hence $\beta$ is surjective. We will
resort to the following preliminary remark: If
$P\in\Prim(q\mcB q)$ and $I$ is an ideal of $\mcB$,
then $I\subset P$ if and only if $qIq\subset qPq$. The
only-if-part is obvious. Suppose $I\not\subset P$. Choose
a simple $\mcB$-module $E$ such that $q\mdot E\neq 0$
and $P=\Ann_{\mcB}(E)$. If $a\in I$ and $a\not\in P$, then
$q\mcB a\mcB q\subset qIq$ and $q\mcB a\mcB q\mdot E=%
q\mdot E\neq 0$ which proves $qIq\not\subset qPq$.
In particular the preceding remark shows that $\beta$
is injective. Furthermore it is easy to see that
$\beta$ is continuous: If $A'\subset\Prim(q\mcB q)$
is closed and $P\in\beta^{-1}(A')\clos$, then
$P\supset\cap\{Q:qQq\in A'\}$ and hence
$qPq\supset Q'$ for all $Q'\in A$. This shows
$qPq\in\overline{A'}=A'$ and thus $P\in\beta^{-1}(A')$.
Finally we prove that $\beta$ is a closed map: Suppose
that $A$ is a closed subset and $P$ an element of
$\Prim\mcB\setminus h(\mcB q\mcB)$ such that
$qPq\in\beta(A)\clos$ which means
$qPq\supset\cap\{qQq:Q\in A\}\supset qk(A)q$. Now
the preliminary remark implies $P\supset k(A)$ and
thus $P\in\overline{A}=A$ and $\beta(P)\in\beta(A)$.
This finishes our proof.
\end{enumerate}
\end{proof}}

Subalgebras of the form $q\mcB q$ for hermitian
idempotents $q\in\mcB^b$ are called corners.\\

The representation theory of exponential Lie groups
is dominated by the fact that certain subquotients
$q\ast(L^1(G)/I)\ast q$ of the group algebra turn
out to be isomorphic to (twisted) weighted
convolution algebras on abelian groups.\\\\
In this context the smooth terminology of twisted
covariance algebras ($L^1$-version) is profitable,
compare~\cite{Green1}, \cite{Rief1}. By definition
a twisted covariance system $(G,\mcA,\tau)$ consists
of (1) a locally compact group $G$ acting strongly
continuously on a Banach $\ast$-algebra~$\mcA$ by
isometric $\ast$-isomorphisms and (2) a twist $\tau$
defined on a closed normal subgroup $H$ of $G$
(i.e.\ a strongly continuous group homomorphism of
$H$ into the group of unitaries of the adjoint algebra
$\mcA^b$ of $\mcA$) such that $\tau(h^x)=\tau(h)^x$
and $a^h=\tau(h)^\ast a\,\tau(h)$ for all $x\in G$,
$h\in H$, and $a\in\mcA$. Let $\mcC_0(G,\mcA,\tau)$
denote the space of all continuous functions
$f:G\to\mcA$ such that $f(xh)=\tau(h)^\ast f(x)$ for
all $x\in G$, $h\in H$ and such that $f$ has compact
support modulo $H$. The closure $L^1(G,\mcA,\tau)$
of $\mcC_0(G,\mcA,\tau)$ with respect to the norm
$|f|_1=\int_{G/H}|f(x)|\,d\dot{x}$ is a Banach
$\ast$-algebra with convolution and involution
given by
\[(f\ast g)(x)=\int_{G/H}f(xy)^{y^{-1}}g(y^{-1})\;d\dot{y}
\quad,\quad f^\ast(x)=\Delta_{G/H}(x^{-1})\,
(f(x^{-1})^\ast)^x\;.\]
A covariance pair $(\pi,\gamma)$ is a unitary
representation $\pi$ of $G$ and a
$\ast$-representation~$\gamma$ of~$\mcA$ in the
same Hilbert space $\frakH$ such that
$\gamma(a^x)=\pi(x)^\ast\gamma(a)\pi(x)$ and
$\gamma(\tau(h))=\pi(h)$. It is well-known that
covariance pairs $(\pi,\gamma)$ correspond to
$\ast$-representations of the twisted
covariance algebra $L^1(G,\mcA,\tau)$.

\begin{defn}\label{eLg_defn:compatible_subspaces}
Let $(G,\mcA,\tau)$ be a twisted covariance system.
A family $\{\mcA_x:x\in G\}$ of closed subspaces
of $\mcA$ is said to be compatible with
$(G,\mcA,\tau)$ if $\tau(h)^\ast \mcA_x=\mcA_{xh}$,
$(\mcA_{xy})^{y^{-1}}\mcA_{y^{-1}}\subset\mcA_x$,
and $((\mcA_{x^{-1}})^\ast)^x=\mcA_x$ for all
$x,y\in G$ and $h\in H$.
\end{defn}

If $\{\mcA_x:x\in G\}$ is compatible, then
$\mcC_0(G,\mcA_x,\tau)=\{f\in\mcC_0(G,\mcA,\tau):%
f(x)\in\mcA_x\}$ defines a subalgebra of
$\mcC_0(G,\mcA,\tau)$. This bears a meaning only
if the $\mcA_x$ are chosen continuously so that
$\mcC_0(G,\mcA_x,\tau)$ and hence its closure
$L^1(G,\mcA_x,\tau)$ are non-zero. One might
think of $\{\mcA_x:x\in G\}$ as a 'bundle'
over $G$ and ask for trivializations.

\begin{defn}\label{eLg_defn:trivialization}
Let $(G,\mcA,\tau)$ be a twisted covariance system
and $\{\mcA_x:x\in G\}$ a compatible family of
one-dimensional subspaces of $\mcA$. We say that a
continuous function $v:G\to\mcA$ is a trivialization
for $\{\mcA_x:x\in G\}$ if $v(x)\in\mcA_x$, $|v(x)|\ge 1$,
$v(xh)=\tau(h)^\ast v(x)$, $v(xy)^{y^{-1}}v(y^{-1})=v(x)$,
and $(v(x^{-1})^\ast)^x=v(x)$ for $x,y\in G$, $h\in H$.
\end{defn}

\begin{prop}\label{eLg_prop:subalgebra_is_Beurling}
Let $(G,\mcA,\tau)$ be a twisted covariance system.
If $v$ is a trivialization for the compatible family
$\{\mcA_x:x\in G\}$ of one-dimensional subspaces
of $\mcA$, then the subalgebra $L^1(G,\mcA_x,\tau)$
is isomorphic to the Beurling algebra $L^1(G/H,w)$
given by the symmetric weight function
$w(\dot{x})=|v(x)|$.
\end{prop}
\begin{proof}
One checks easily that $\Phi(b)(x)=b(\dot{x})v(x)$
defines an isometric isomorphism from $L^1(G/H,w)$
onto $L^1(G,\mcA_x,\tau)$.
\end{proof}

Let $q\in\mcA\subset L^1(G,\mcA,\tau)^b$ be a hermitian
idempotent. Since $(q\ast f\ast q)(x)=q^xf(x)q$ for
all $f\in L^1(G,\mcA,\tau)$, it follows
$q\ast L^1(G,\mcA,\tau)\ast q=L^1(G,q^x\!\mcA\,q,\tau)$.
In Theorem~\ref{eLg_thm:compute_subquotients} we
treat the case where $q$ is minimal and the
$q^x\!\mcA\,q$ are one-dimensional.\\

The following theorem is due to Poguntke, see part~(4)
and~(5) of the proof of the main theorem in~\cite{Pog3}.
The idea goes back to Theorem~5 of Leptin and Poguntke
in~\cite{Lep4}.

\begin{thm}\label{eLg_thm:compute_subquotients}
Let $(\pi,\gamma)$ be a covariance pair of the twisted
covariance system $(G,\mcA,\tau)$ such that $\gamma$ is
irreducible and faithful. Suppose that there exists a
minimal hermitian idempotent $q\in\mcA$. Then the corner
$q\ast L^1(G,\mcA,\tau)\ast q=L^1(G,q^x\!\mcA\,q,\tau)$
is isometrically isomorphic to a weighted Beurling
algebra $L^1(G/H,w)$ where $w$ is a symmetric weight
function on $G/H$.
\end{thm}
\begin{proof}
By Lemma~\ref{eLg_lem:algebras_with_minimal_hermitian_%
idempotents}.\textit{(i)} there exists a unit vector
$\lambda\in\frakH$ such that $\gamma(q)\xi=%
\langle\xi,\lambda\rangle\lambda$. Now $\gamma(q^xaq)=%
\pi(x)^\ast\gamma(q)\pi(x)\gamma(a)\gamma(q)$ implies
\begin{equation*}\label{eLg_equ:evaluate_gamma}
\gamma(q^xaq)\xi=\langle\pi(x)\gamma(a)\lambda,\lambda
\rangle\;\langle\xi,\lambda\rangle\;\pi(x)^{-1}\lambda\;.
\end{equation*}
For every $x\in G$ there exists some $a\in\mcA$ such that
$\langle\pi(x)\gamma(a)\lambda,\lambda\rangle$ is
non-zero because $\gamma$ is irreducible. This shows
$\gamma(q^x\mcA q)=\mC\,\pi(x)^{-1}\gamma(q)$
so that $q^x\mcA q$ is one-dimensional. There
is a unique element $v(x)\in\mcA$ such that
$\gamma(v(x))\xi=\langle\xi,\lambda\rangle%
\pi(x)^{-1}\lambda$. Clearly $v:G\to\mcA$ is
continuous and $|v(x)|\ge |\gamma(v(x))|=1$.
Further one computes
\begin{align*}
\gamma(v(xh))&=\pi(h)^\ast\gamma(v(x))=
\tau(h)^\ast\gamma(v(x))\\
\gamma(v(x))&=\gamma\left(v(xy)^{y^{-1}}v(y^{-1})\right)\\
\gamma(v(x))&=\gamma((v(x^{-1})^\ast )^x)
\end{align*}
which proves that $v$ is a trivialization for
$\{q^x\!\mcA\,q:x\in G\}$ because $\gamma$ is faithful.
Now Proposition~\ref{eLg_prop:subalgebra_is_Beurling}
gives the desired result.
\end{proof}

Our aim is to apply Theorem~\ref{eLg_thm:compute_%
subquotients} to certain quotients of group algebras:
Let $H$ be a closed normal subgroup of $G$. It is
known that $\LG$ is isomorphic to the twisted
covariance algebra $L^1(G,\LH,\tau)$ with
$G$-action $a^x\,(h)=\delta_H(x^{-1})\,a(h^{x^{-1}})$
and twist $\tau(k)a\,(h)=a(k^{-1}h)$, compare the
corollary to Proposition~1 in~\cite{Green1}. Suppose
that $\gamma$ is an irreducible representation of $H$.
The crucial assumption in Theorem~\ref{eLg_thm:%
compute_subquotients} is that $\gamma$ can be completed
to a covariance pair $(\pi,\gamma)$ of $(G,\LH,\tau)$,
or equivalently, that it can be extended to a
representation $\pi$ of $G$, which is only possible if
$x\mdot\gamma$ is unitarily equivalent to $\gamma$ for
all $x\in G$. If such a $\pi$ exists, then the ideal
$I'=\ker_{\LH}\gamma$ is $G$- and $\tau$-invariant so
that the covariance algebra $L^1(G,\LH/I',\dot{\tau})$
with induced $G$-action and twist is well-defined. This
algebra is isomorphic to the quotient $\LG/I$ where
$I=\ind_H^G(I')$. To apply Theorem~\ref{eLg_thm:%
compute_subquotients} it remains to find minimal
hermitian idempotents in $\LH/I'$.\\

The existence of an extension $\pi$ of $\gamma$ is
guaranteed under the assumptions of

\begin{prop}\label{eLg_prop:extension_of_gamma}
Let $G$ be an exponential Lie group, $f\in\frakg^\ast$,
and $\pi=\mcK(f)\in\widehat{G}$. Suppose that $\frakh$ is
an ideal of $\frakg$ such that $\frakg=\frakg_f+\frakh$.
Let $f_0=f\,|\,\frakh$ and $\gamma=\mcK(f_0)$. Then
$\pi\,|\,H$ is unitarily equivalent to $\gamma$. This
means that $\pi$ yields an extension of $\gamma$.
\end{prop}
\begin{proof}
Recall that there exists a $\frakg_f$-invariant Pukanszky
polarization $\frakp_0\subset\frakh$ at $f_0\in\frakh^\ast$
because $\frakg$ is exponential, see \S 4, Chapter I
of~\cite{LepLud} and Chapter 5 of \cite{Bern_et_alii}. We
shall verify that $\frakp=\frakg_f+\frakp_0\subset\frakg$
defines a Pukanszky polarization at $f\in\frakg^\ast$: Clearly
$[\frakp,\frakp]\subset[\frakg_f,\frakg_f]+[\frakg_f,%
\frakp_0]+[\frakp_0,\frakp_0]\subset\frakp\cap\ker f$.
Note that $\frakp_0=\frakp\cap\frakh$ and
$\frakh_{f_0}=\frakg_f\cap\frakh$. Using the canonical
isomorphisms $\frakh/\frakh_{f_0}\cong\frakg/\frakg_f$ and
$\frakp_0/\frakh_{f_0}\cong\frakp/\frakg_f$ we conclude
that $\dim\frakg/\frakg_f=\frac{1}{2}\dim\frakp/\frakg_f$.
It remains to prove that $\coAd(P)f=f+\frakp^\bot$. If
$h\in\frakp^\bot$, then $h_0=h\,|\,\frakh\in\frakp_0^\bot%
\subset\frakh^\ast$. Since $\frakp_0$ is a Pukanszky
polarization at $f_0$, there exists some $x\in P_0$ such
that $\coAd(x)f_0=f_0+h_0$. This implies $\coAd(x)f=f+h$
because $\frakg=\frakg_f+\frakh$. Thus
$\coAd(P)f=f+\frakp^\bot$.\\\\
Fix a relatively $G$-invariant measure on $G/P$. There
exists a unique relatively $H$-invariant measure on $H/P_0$
such that the canonical $H$-equivariant diffeomorphism
$H/P_0\to G/P$ is measure-preserving. The modular functions
of these measures satisfy $\Delta_{G,P}\,|\,H=\Delta_{H,P_0}$.
Now it follows that $\varphi\mapsto \varphi\,|\,H$ defines
a unitary isomorphism from $L^2(G,\chi_f)$ onto
$L^2(H,\chi_{f_0})$ which intertwines $\pi\,|\,H$
and $\gamma$. This completes the proof.
\end{proof}

The next theorem results from the achievements of
Poguntke in \cite{Pog4} concerning the parametrization
of simple $\LG$-modules. In \cite{LudMol} Ludwig and
Molitor-Braun gave a simplified proof of Theorem~%
\ref{eLg_thm:commutative_subquotients} which in
particular avoids projective representations. The
decisive idea of Ludwig and Molitor-Braun may
be recapitulated as follows: If $H/N$ is chosen
to be a vector space complement to $M/N$ instead
of $K/N$ as in \cite{Pog4}, then one ends up
directly with a commutative subquotient.\\

Recall that any simple $\LG$-module can be regarded
as an $\LG^b$-module. In particular, if $N$ is a
closed subgroup of $G$, then $E$ becomes an
$\LN$-module so that $\Ann_{\LN}(E)$ is
defined.\\\\
Let $(G,\mcA)$ be a covariance system. A
$G$-invariant ideal $J$ of $\mcA$ is called
$G$-prime if $J_1J_2\subset J$ for $G$-invariant
ideals $J_1$, $J_2$ of $\mcA$ implies
$J_1\subset J$ or $J_2\subset J$. If $N$ is a
closed normal subgroup of $G$, then one has the
covariance system $(G,\LN)$ with $G$-action
$a^x(n)=\delta_N(x^{-1})\,a(n^{x^{-1}})$.

\begin{thm}\label{eLg_thm:commutative_subquotients}
Let $N$ be a closed, connected, coabelian, nilpotent
subgroup of the exponential Lie group $G$.
\begin{enumerate}
\item If $E$ is a simple $\LG$-module, then
$J=\Ann_{\LN}(E)$ is $G$-prime.
\item Conversely let $J$ be a $G$-prime ideal of $\LN$.
Define $I=\ind_N^G(J)$. The simple $\LG$-modules $E$
such that $J\subset\Ann_{\LN}(E)$ are in a canonical
bijection with the simple modules of $\mcB=L^1(G)/I$.
Moreover, there exist hermitian idempotents
$q\in\mcB^b$ such that the corner $q\ast\mcB\ast q$
is commutative and such that $q\mdot E\neq 0$
exactly for those simple $\mcB$-modules $E$
with $J=\Ann_{\LN}(E)$.
\end{enumerate}
\end{thm}
\begin{proof}\hfill
\begin{enumerate}
\item[\textit{1.}] Recall that $\lambda(x)f\,(y)=f(x^{-1}y)$
defines a group homomorphism from $G$ into the unitary
group of $\LG^b$. Since $a^x\mdot\xi=%
\lambda(x^{-1})\mdot(a\mdot(\lambda(x)\mdot\xi)))$
for $a\in\LN$, $x\in G$, and $\xi\in E$, it
follows that $J$ is $G$-invariant. Now let
$J_1$, $J_2$ be $G$-invariant ideals of $\LN$
such that $J_1\ast J_2\subset J=\Ann_{\LN}(E)$.
Then
\begin{align*}
\ind_N^G(J_1)\ast\ind_N^G(J_2)&\subset
\left(\;\LG\ast J_1\ast\LG\ast J_2\ast\LG\;\right)\clos\\
&\subset\left(\;\LG\ast J_1\ast J_2\ast\LG\;\right)\clos
\subset\Ann_{\LG}(E)\;.
\end{align*}
The first inclusion is obvious and the second one
results from Lemma~\ref{eLg_lem:convolution_%
with_Ginvariant_ideals}. For the third one we use the
fact that $\Ann_{\LG}(E)$ is closed. Since this ideal is
prime, it follows $\ind_N^G(J_k)\subset\Ann_{\LG}(E)$
for $k=1$ or $2$. Finally we obtain
$J_k\subset\Ann_{\LN}(E)$ because $E$
is a simple $\LG$-module.
\item[\textit{2.}] Let $J$ be a $G$-prime ideal
of $\LN$ and $I=\ind_N^G(J)$. In order to prove
the first assertion of~\textit{2.}\ it suffices to
verify that $J\subset\Ann_{\LN}(E)$ if and only if
$I\subset\Ann_{\LG}(E)$. The only-if part is
obvious. Suppose that $I\subset\Ann_{\LG}(E)$.
Then $\LG\mdot(J\mdot E)\subset\ind_N^G(J)\mdot E=%
I\mdot E=0$ implies $J\mdot E=0$ because $E$ is
simple. This means $J\subset\Ann_{\LN}(E)$.\\\\
Next we prove the existence of appropriate hermitian
idempotents in the adjoint algebra of $\mcB=\LG/I$:
Generalizing a theorem of Poguntke in~\cite{Pog4},
Ludwig and Molitor-Braun proved in Theorem~1.1.6
of~\cite{LudMol} that there exists a unique
orbit $G\mdot\sigma$ in $\widehat{N}$ such that
$J=k(G\mdot\sigma)$. Since the Kirillov map of $N$
is bijective, we can choose $l\in\frakn^\ast$ such
that $\mcK(l)=\sigma$ and $f\in\frakg^\ast$ such that
$f\,|\,\frakn=l$. We stress that the definition of
the stabilizer $\frakm=\frakg_f+\frakn$ depends
only on the orbit $\coAd(G)l$, i.e., on the
$G$-prime ideal $J$. As usual $M$ denotes the
closed, connected subgroup of~$G$ with Lie
algebra~$\frakm$. In addition we fix a closed,
connected subgroup $H$ of~$G$ containing $N$
such that $H/N$ is complementary to $M/N$ in the
vector space~$G/N$. In particular $G=G_fH$.\\\\
The ideal $I'=\ind_N^H(J)$ is invariant under the
$G$-action $b^x(h)=\delta_H(x^{-1})\,b(h^{x^{-1}})$
and the twist $\tau(k)b\,(h)=b(k^{-1}h)$ in $\LH$.
Since $I=\ind_H^G(I')$, the quotient $\mcB=L^1(G)/I$
can be identified with $L^1(G,\LH/I',\dot{\tau})$.\\\\
Let $f_0=f\,|\,\frakh\in\frakh^\ast$ and $\gamma=%
\mcK(f_0)\in\widehat{H}$. Since $\frakg=\frakg_f+\frakh$,
Proposition~\ref{eLg_prop:extension_of_gamma} implies that
$\pi=\mcK(f)$ furnishes an extension of~$\gamma$. Note
that $\frakh_{f_0}=\frakg_f\cap\frakh\subset\frakn$.
This shows that $\coAd(H)f_0$ is saturated over
$\frakn$, i.e., $\coAd(H)f_0=f_0+\frakn^\bot$.
In particular $\ker_{\LH}\gamma$ is invariant under
multiplication by characters of $H/N$,
and hence $\mcC_\infty(H/N)$-invariant. Now
Theorem~\ref{eLg_thm:induced_ideals_in_L1}
implies $\ker_{\LH}\gamma=\ind_N^H(J)=I'$
because $\ker_{\LN}\gamma=k(G\mdot\sigma)=J$
by Lemma~\ref{eLg_lem:restriction}. We have shown
that $\gamma$ yields a faithful irreducible
representation of $\mcA=\LH/I'$ which admits an
extension~$\pi$.\\\\
Let us fix an arbitrary minimal hermitian idempotent
$q\in\mcA=\LH/I'$. Since $H$ is an exponential Lie
group, the existence such idempotents is guaranteed
by Poguntke's results in~\cite{Pog3}. Finally
Theorem~\ref{eLg_thm:compute_subquotients} shows
that the corner $q\ast\mcB\ast q$ is commutative
as it is isomorphic to a weighted Beurling algebra
$L^1(G/H,w)$ on the commutative group $G/H$.\\\\
Let $E$ be a simple $\LG$-module such that
$J\subset\Ann_{\LN}(E)$. It remains to be shown that
$q\mdot E\neq 0$ if and only if $J=\Ann_{\LN}(E)$.
The subsequent proof of the if-part is from~\cite{Pog4}.
We begin with a preliminary remark: Let $Q'_0$ denote
the ideal of all $b\in\LH$ such that $\gamma(b)$
has finite rank. Clearly $I'\subset Q'_0$ and
$q\in Q'_0\setminus I'$. Let $\frakH$ denote the
representation space of $\gamma$, and $F$ the simple
module associated to $\gamma$ in the sense of
Lemma~\ref{eLg_lem:associated_simple_modules}. It is
known that $\gamma(Q'_0)$ is equal to the algebra of
all finite rank operators $A$ of $\frakH$ such that
$A(\frakH)\subset F$ and $A^\ast(\frakH)\subset F$,
compare Th\'eor\`eme~\textit{2.}\ of Dixmier
in~\cite{Dix2}. From this we deduce
$(Q'_0/I')\ast b\ast(Q'_0/I')=Q'_0/I'$ for
all $b\in Q'_0\setminus I'$. In particular we
see that either $Q'_0\subset\Ann_{\LH}(E)$ or
$b\mdot E\neq 0$ for all
$b\in Q'_0\setminus I'$.\\\\
Suppose that $J=\Ann_{\LN}(E)$. Since $\coAd(H)f_0$
is saturated over $\frakn$, it follows that
$\gamma\otimes\alpha$ is unitarily equivalent to
$\gamma$ for all $\alpha\in(H/N)\widehat{\;\;}$.
Thus $Q'_0$ and hence its closure $Q'$ are
$(H/N)\widehat{\;\;}$-invariant. By
Theorem~\ref{eLg_thm:induced_ideals_in_L1} 
there exists an $H$-invariant ideal $R$ of $\LN$
such that $Q'=\ind_N^H(R)$. Note that $R$ is not
contained in $J$. Thus $Q=\ind_N^G(R)=\ind_H^G(Q')$
is not contained in $\Ann_{\LG}(E)$ because $E$
is simple. Consequently $Q_0'$ is not contained
in $\Ann_{\LH}(E)$ so that $q\mdot E\neq 0$ by
the preliminary remark.\\\\
In order to prove the only-if-part we suppose
$q\mdot E\neq 0$. The preceding remark implies
$Q'_0\cap\Ann_{\LH}(E)\subset I'$. Now we conclude
$Q'\ast\Ann_{\LH}(E)\subset I'$. Since $I'$
is $G$-prime and $Q'$, $\Ann_{\LH}(E)$ are
$G$-invariant ideals, we get $\Ann_{\LH}(E)\subset I'$
because $Q'\not\subset I'$. Let $a\in\Ann_{\LN}(E)$.
Since $\LH\ast a\ast\LH$ is contained in
$I'=\ker_{\LH}\gamma$, it follows
$a\in\ker_{\LN}\gamma=k(G\mdot\sigma)=J$
because $\gamma$ is irreducible.
\end{enumerate}
\end{proof}

Let $J$ a given $G$-prime ideal of $\LN$. In
combination with part \textit{4.}\ of 
Proposition~\ref{eLg_prop:corners} the
preceding theorem shows that the equivalence
classes of all simple $\LG$-modules $E$ with
annihilator $\Ann_{\LN}(E)=J$ are in a one-to-one
correspondence with the characters of the
commutative Beurling algebra
$q\ast(\LG/I)\ast q\cong L^1(G/H,w)$.\\\\
Here we are content with this rough description
and deliberately renounce more delicate questions
such as obtaining estimates for the weight $w$,
which can be found in~\cite{Pog4}.

\begin{cor}\label{eLg_cor:simple_modules_with_the_%
same_annihilator_on_N}
If $E$, $F$ are simple $\LG$-modules with
$\Ann_{\LG}(E)\subset\Ann_{\LG}(F)$ and
$J=\Ann_{\LN}(E)=\Ann_{\LN}(F)$, then $E$
and $F$ are isomorphic.
\end{cor}
\begin{proof}
Note that $E$ and $F$ can be regarded as $\mcB$-modules
where $\mcB=\LG/\ind_N^G(J)$. Let $q\in\mcB^b$ be a
hermitian idempotent as in part~\textit{2.}\ of
Theorem~\ref{eLg_thm:commutative_subquotients}.
By definition $q\mdot E$ and $q\mdot F$ are non-zero.
Hence they are simple modules over the commutative
$\ast$-algebra $q\ast\mcB\ast q$ with annihilators
$\Ann_{q\ast\mcB\ast q}(q\mdot E)\subset
\Ann_{q\ast\mcB\ast q}(q\mdot F)$, compare
Proposition~\ref{eLg_prop:corners}. It results from
Schur's Lemma that $q\mdot E$ and $q\mdot F$ are
one-dimensional, have the same annihilator, and
are thus isomorphic. Finally Proposition~%
\ref{eLg_prop:corners}.\textit{4.}\ shows that
$E$ and $F$ are isomorphic as $\mcB$-modules,
and also as $\LG$-modules.
\end{proof}

\begin{cor}\label{eLg_cor:representations_over_same%
_orbit_on_N}
If $\pi,\rho\in\widehat{G}$ are irreducible such
that $\ker_{\LG}\pi\subset\ker_{\LG}\rho$ and
$\ker_{\LN}\pi=\ker_{\LN}\rho$, then $\pi$ and
$\rho$ are unitarily equivalent.
\end{cor}
\begin{proof}
Let $E$, $F$ denote the simple $\LG$-modules associated
to $\pi$, $\rho$ respectively in the sense of Lemma~%
\ref{eLg_lem:associated_simple_modules}.\textit{(ii)}.
By definition $\Ann_{\LG}(E)\subset\Ann_{\LG}(F)$
and $J=\Ann_{\LN}(E)=\Ann_{\LN}(F)$. Thus $E$ and $F$
are isomorphic by Corollary~\ref{eLg_cor:simple_%
modules_with_the_same_annihilator_on_N}. This means
$\ker_{\LG}\pi=\Ann_{\LG}(E)=\Ann_{\LG}(F)=\ker_{\LG}\rho$.
Finally $\pi$ and $\rho$ are unitarily equivalent by
Lemma~\ref{eLg_lem:algebras_with_minimal_%
hermitian_idempotents}.\textit{(ii)}.
\end{proof}

These preparations make it easy to prove the
main result of this section.

\begin{prop}\label{eLg_prop:closed_orbit_is_sufficient}
Let $G$ be an exponential Lie group and $N$ a closed,
connected, coabelian, nilpotent subgroup. Let
$\pi\in\widehat{G}$ and $G\mdot\sigma$ be the
unique $G$-orbit in $\widehat{N}$ such that
$k(G\mdot\sigma)=\ker_{\CN}\pi$. If $G\mdot\sigma$
is closed in $\widehat{N}$, then $\ker_{\CG}\pi$
is $\LG$-determined.
\end{prop}
\begin{proof}
Let $\rho$ be in $\widehat{G}$ such that
$\ker_{\LG}\pi\subset\ker_{\LG}\rho$. Restricting
to~$N$ we obtain $\ker_{\LN}\pi\subset\ker_{\LN}\rho$.
Since $N$ is $\ast\,$-regular as a connected
nilpotent Lie group, it follows
$k(G\mdot\sigma)=\ker_{\CN}\pi\subset\ker_{\CN}\rho$.
This yields $\ker_{\CN}\pi=\ker_{\CN}\rho$ because
the orbit $G\,\mdot\,\sigma$ is closed. Finally
Corollary~\ref{eLg_cor:representations_over_same_%
orbit_on_N} implies that $\pi$ and $\rho$ are
unitarily equivalent so that in particular
$\ker_{\CG}\pi=\ker_{\CG}\rho$.
\end{proof}

However, the preceding results are limited to
the case when $G\mdot\sigma$ is closed
in $\widehat{N}$.

\begin{rem}
Let $N$ be a coabelian, nilpotent subgroup of $G$
and $\pi\in\widehat{G}$ such that $G\mdot\sigma$
is not closed in $\widehat{N}$. To prove that
$\ker_{\CG}\pi$ is $\LG$-determined, one must
show that $\ker_{\CG}\pi\not\subset\ker_{\CG}\rho$
implies $\ker_{\LG}\pi\not\subset\ker_{\LG}\rho$
for all $\rho\in\widehat{G}$. Note that
$J=\ker_{\LN}\pi$ is $G$-prime and define
$I=\ind_N^G(J)$. To avoid trivialities we can
assume $\ker_{\LN}\pi\subset\ker_{\LN}\rho$ so
that $\pi$ and $\rho$ factor to representations
of $\mcB=\LG/I$. In addition we suppose
that $\ker_{\LN}\pi\neq\ker_{\LN}\rho$. Such
representations~$\rho$ are likely to exist if
$G\mdot\sigma$ is not closed. If $q\in\mcB^b$ is
a hermitian idempotent as in
Theorem~\ref{eLg_thm:commutative_subquotients},
then $\rho(q)=0$. This means that restriction
to the subquotient $q\mcB q$ is not appropriate
for proving $\ker_{\LG}\pi\not\subset\ker_{\LG}\rho$
in this case. 
\end{rem}

\boldmath
\section{A strategy to prove primitive
 $\ast\,$-regularity}
\unboldmath\label{eLg_sec:strategy}
Let $G$ be an exponential Lie group and $\frakn$ a
coabelian nilpotent ideal of its Lie algebra~$\frakg$.
In order to prove that $G$ is primitive $\ast\,$-regular,
one must show that $\ker_{\CG}\pi$ is $\LG$-determined
for all $\pi\in\widehat{G}$, i.e., according to
Definition~\ref{Bsa_def:A_determined} one must prove
that
$$\ker_{\CG}\pi\not\subset\ker_{\CG}\rho\quad\text%
{implies}\quad\ker_{\LG}\pi\not\subset\ker_{\LG}\rho$$
for all $\rho\in\widehat{G}$. Let $f,g\in\frakg^\ast$
such that $\pi=\mcK(f)$ and $\rho=\mcK(g)$. Since the
Kirillov map of $G$ is a homeomorphism with respect
to the Jacobson topology on the primitive ideal space
$\Prim\CG$ and the quotient topology on the coadjoint
orbit space $\frakg^\ast/\coAd(G)$, the relation for
the $C^\ast\,$-kernels is equivalent to
$\coAd(G)g\not\subset(\coAd(G)f)\clos$. From the
preceding subsections we extract the following
observations:
\begin{enumerate}
\item Let $\fraka$ be a non-trivial ideal of $\frakg$
such that $f=0$ on $\fraka$. Let $A$ be the connected
subgroup of $G$ with Lie algebra $\fraka$. Since
$\pi=1$ on $A$, we can pass over to a representation
$\dot{\pi}$ of the quotient $\dot{G}=G/A$. It follows
from Lemma~\ref{Bsa_lem:quotients_of_algebras} that
$\ker_{\CG}\pi$ is $\LG$-determined if and only
if $\ker_{C^\ast(\dot{G})}\dot{\pi}$ is
$L^1(\dot{G})$-determined. Often $\dot{G}$ is known
to be primitive $\ast\,$-regular by induction. If this
is the case for all proper quotients $\dot{G}$ of $G$,
then we can assume that $f$ is in general position,
i.e., $f\neq 0$ on all non-trivial ideals $\fraka$
of $\frakg$.
\item If the stabilizer $\frakm=\frakg_f+\frakn$ is
nilpotent, then $\ker_{\CG}\pi$ is $\LG$-determined by
Propositions~\ref{eLg_prop:induced_from_regular_subgroup}
and~\ref{eLg_prop:ideal_induced_from_M} because $M$ is
$\ast$-regular.
\item If $\frakg=\frakg_{f'}+\frakn$, then
$\ker_{\CG}\pi$ is $\LG$-determined by Proposition~%
\ref{eLg_prop:closed_orbit_is_sufficient}. Here and
in the sequel $f'$ denotes the restriction of $f$
to $\frakn$.
\item If $\coAd(G)g'$ is not contained in the closure
of $\coAd(G)f'$, then it follows $\ker_{\CN}\pi\not%
\subset\ker_{\CN}\rho$ because the Kirillov map is a
homeomorphism. Since $N$ is $\ast\,$-regular, we
obtain $\ker_{\LN}\pi\not\subset\ker_{\LN}\rho$ and
hence $\ker_{\LG}\pi\not\subset\ker_{\LG}\rho$.
\end{enumerate}

\begin{lem}\label{eLg_lem:1codim_nilpotent_ideal}
If there exists a one-codimensional nilpotent ideal
$\frakn$ of $\frakg$, then $G$ is primitive
$\ast\,$-regular.
\end{lem}
\begin{proof}
Let $f\in\frakg^\ast$ be arbitrary and $\pi=\mcK(f)$.
The assumption $\dim\frakg/\frakn=1$ implies that
either $\frakg=\frakg_{f'}+\frakn$, or that
$\frakm=\frakg_f+\frakn=\frakn$ is nilpotent.
Clearly the preceding remarks show that
$\ker_{\CG}\pi$ is $\LG$-determined.
\end{proof}

\begin{defn}\label{eLg_def:g_critical_for_orbit}
Let $f\in\frakg^\ast$ be in general position. As
before $r:\frakg^\ast\onto\frakn^\ast$ is given by
$r(g)=g'=g\,|\,\frakn$. We define $\Omega$ as the
$r$-preimage of the closure of $\coAd(G)f'$
in~$\frakn^\ast$. Note that $\Omega$ is a closed
subset of $\frakg^\ast$ containing $\coAd(G)f$ and
that $g\in\Omega$ if and only if $g'$ is in the
closure of~$\coAd(G)f'$. We say that $g$ is
critical for the orbit $\coAd(G)f$ if
$g\in\Omega\setminus(\coAd(G)f)\clos$. By
Proposition~\ref{eLg_prop:closed_orbit_is_sufficient}
we can even assume $\coAd(G)g'\neq\coAd(G)f'$.
\end{defn}

In order to prove the primitive $\ast\,$-regularity
of $G$ it thus suffices to verify the following
two assertions:
\begin{enumerate}
\item Every proper quotient $\dot{G}$ of $G$ is
primitive $\ast\,$-regular.
\item If $f\in\frakg^\ast$ is in general position
such that the stabilizer $\frakm=\frakg_f+\frakn$
is a proper, non-nilpotent ideal of $\frakg$ and if
$g\in\frakg^\ast$ is critical for the orbit
$\coAd(G)f$, then it follows
$$\ker_{\LG}\pi\not\subset\ker_{\LG}\rho\;.$$
\end{enumerate}

Let $d_1,\ldots,d_m$ be a coexponential basis
for $\frakm$ in $\frakg$. We obtain a diffeomorphism
from $\mR^m$ onto~$G/M$ by composing the smooth map
$E(s)=\exp(s_1d_1)\mdot\ldots\mdot\exp(s_md_m)$ with the
quotient map $G\onto G/M$. Define $\tilde{f}=f\,|\,\frakm$,
$\tilde{f}_s=\coAd(E(s))\tilde{f}$ in $\frakm^\ast$,
and $\tilde{\pi}_s=\mcK(\tilde{f}_s)$ in $\widehat{M}$.\\\\
Two properties of $\pi$ and their counterpart in the
Kirillov picture are worth mentioning. First $\pi\,|\,M$
is reducible. By Lemma \ref{eLg_lem:restriction} we
know that $\pi\,|\,M$ is weakly equivalent to the
subset $\{\tilde{\pi}_s:s\in\mR^m\}$ of $\widehat{M}$.
In the orbit picture $\coAd(G)\tilde{f}$ decomposes
into the disjoint union of the orbits
$\{\,\coAd(M)\tilde{f}_s:s\in\mR^m\,\}$.\\\\
Secondly, $\ker_{\CG}\pi$ is induced from $M$ by
Proposition~\ref{eLg_prop:ideal_induced_from_M}. Hence
$\ker_{\CG}\pi\subset\ker_{\CG}\rho$ is equivalent to
the corresponding inclusion in $\CM$. The same holds
true in $\LM$. In the Kirillov picture we have
$\coAd(G)f=\coAd(G)f+\frakm^\bot$ so that $g$ is
in $(\coAd(G)f)\clos$ if and only if $\tilde{g}$ is
in the closure of $\coAd(G)\tilde{f}$.\\\\
In analogy to Definition~\ref{eLg_def:g_critical_for_orbit}
we define $\tilde{\Omega}\subset\frakm^\ast$ and critical
$\tilde{g}$ for the orbit $\coAd(G)\tilde{f}$
in~$\frakm^\ast$. We say that $\tilde{f}$ is in
general position if $f(\fraka)\neq 0$ on any
non-trivial ideal $\fraka$ of $\frakg$ such that
$\fraka\subset\frakm$. Now it is easy to see that
we can replace the second assertion by the
following equivalent one:

\begin{enumerate}
\item[3.] Let $\frakm$ be a proper, non-nilpotent
ideal of $\frakg$ such that $\frakm\supset\frakn$.
If $\tilde{f}\in\frakm^\ast$ is in general position
such that $\frakm=\frakm_{\tilde{f}}+\frakn$ and
if $\tilde{g}\in\frakm^\ast$ is critical for the
orbit $\coAd(G)\tilde{f}$, then the relation
\begin{equation}\label{eLg_equ:no_inclusion_in_LM}
\bigcap\limits_{s\in\mR^m}\,\ker_{\LM}\,\tilde{\pi}_s%
\not\subset\ker_{\LM}\,\tilde{\rho}
\end{equation}
holds for the representations
$\tilde{\pi}_s=\mcK(\tilde{f}_s)$
and $\tilde{\rho}=\mcK(\tilde{g})$.
\end{enumerate}

In this situation producing functions $c\in\LM$ such
that $\pi_s(c)=0$ for all $s$ and $\rho(c)\neq 0$
turns out to be a great challenge.

\begin{rem}\label{eLg_rem:zn_contained_in_zm}
The stabilizer $\frakm=\frakg_f+\frakn$ has a
remarkable algebraic property. Note that the ideal
$[\frakm,\frakz\frakn]=[\frakm_f,\frakz\frakn]$ is
contained in $\ker f$. If in addition $f$ is
in general position, then it follows
$[\frakm,\frakz\frakn]=0$ so that
$\frakz\frakn\subset\frakz\frakm$.
\end{rem}

\begin{lem}\label{eLg_lem:metabelian_groups_are_regular}
If $\frakg$ is an exponential Lie algebra such
that $[\frakg,\frakg]$ is commutative, then
$G$ is $\ast$-regular.
\end{lem}
\begin{proof}
Let $f\in\frakg^\ast$ be arbitrary. If $\fraka$ denotes
the largest ideal of $\frakg$ contained in $\ker f$, then
$\dot{f}$ on $\dot{\frakg}=\frakg/\fraka$ is in general
position. By Remark~\ref{eLg_rem:zn_contained_in_zm} we
obtain $\dot{\frakn}=[\dot{\frakg},\dot{\frakg}]=%
\frakz\dot{\frakn}\subset\frakz\dot{\frakm}$. Thus the
quotient $\dot{\frakm}$ of $\frakm=\frakg_f+\frakn$
is 2-step nilpotent so that $\frakg$ satisfies
condition~(R). Now Theorem~\ref{eLg_thm:%
conditionR_implies_regularity} yields the
assertion of this lemma.
\end{proof}

\boldmath
\section{A non-$\ast$-regular example}
\unboldmath
Let $\frakg$ be an exponential Lie algebra of
dimension $\le 5$. In view of Lemma~\ref{eLg_lem:%
1codim_nilpotent_ideal} and~\ref{eLg_lem:metabelian_%
groups_are_regular} we assume that the nilradical
$\frakn$ (the maximal nilpotent ideal) of $\frakg$
is not commutative and of dimension~$\le 3$, i.e.,
$\frakn=\langle e_1,e_2,e_3\rangle$ is a 3-dimensional
Heisenberg algebra. Further we suppose that
$f\in\frakg^\ast$ is general position (so that
$f(e_3)\neq 0$) and that the stabilizer
$\frakm=\frakg_f+\frakn$ is a proper, non-nilpotent
ideal. These assumptions imply that $\frakg$ has a
basis $d,e_0,\ldots,e_3$ satisfying the commutator
relations $[e_1,e_2]=e_3$, $[e_0,e_1]=-e_1$,
$[e_0,e_2]=e_2$, $[d,e_2]=e_2$, and $[d,e_3]=e_3$.
The algebra $\frakg$ and the stabilizer
$\frakm=\langle e_0,\ldots,e_3\rangle$ are
specified as $\frakg=\frakb_5$ and
$\frakm=\frakg_{4,9}(0)$ in the (complete)
list of all non-symmetric Lie algebras up to
dimension~6 in~\cite{Pog2}, whereas symmetry
is equivalent to $\ast$-regularity by
Theorem~10 of~\cite{Pog4}.\\\\
We work in coordinates of the second kind \wrt the
above Malcev basis, given by the diffeomorphism
$\Phi(t,x)=\exp(te_0)\exp(x_1e_1)\exp(x_2e_2+x_3e_3)$
from $\mR^4$ onto $M$. Denote $f\,|\,\frakm$ again
by $f$ and let $f_s=\coAd(\exp(sd))f$. By choosing
an appropriate representative on the orbit $\coAd(G)f$
we can achieve $f(e_3)=1$ and $f(e_1)=f(e_2)=0$. Now
we compute
\begin{align*}
\coAd\left(\,\exp(sd)\,\Phi(t,x)\,\right)f\;(e_0)&=%
f(e_0)-x_1x_2\,,\\
(e_1)&=e^{t}x_2\,,\\
(e_2)&=-e^{-s}e^{-t}x_1\,,\\
(e_3)&=e^{-s}\,.
\end{align*}
These formulas for the coadjoint representation suggest
to define the $\coAd(M)$-invariant polynomial function
$$p=e_0e_3-e_1e_2-f(e_0)e_3$$
on $\frakm^\ast$ such that $p(h)=0$ for all
$h\in X=\coAd(G)f=\cup\{\coAd(M)f_s:s\in\mR\}$. Here
$e_\nu$ is interpreted as the linear function
$e_\nu(h)=h(e_\nu)$ on $\frakm^\ast$ and $f(e_0)$
is a constant. Note that $p$ is even
$\coAd(G)$-semi-invariant. The closure of the
orbit $X=\coAd(G)f$ in $\frakm^\ast$ can be
characterized by means of the $\coAd(M)$-invariant
polynomial $p$.

\begin{lem}
Let $g\in\Omega\subset\frakm^\ast$. Then
$g\in\overline{X}$ if and only if $p(g)=0$.
\end{lem}

\begin{proof}
The only-if-part is obvious because $p(h)=0$ for
all $h\in X$. Let us prove the opposite direction.
Let $g\in\Omega$ such that $p(g)=0$. We must
distinguish four cases. First we assume $g(e_3)\neq 0$.
Since $g\in\Omega$, it follows $g(e_3)>0$. Without
loss of generality we can assume $g(e_3)=1=f(e_3)$
and $g(e_1)=g(e_2)=0$. Now $p(g)=0$ implies
$g(e_0)=f(e_0)$ so that $g\in X$. Next we consider
the case $g(e_3)=g(e_2)=0$ and $g(e_1)\neq 0$. If
we define $s_n=n$, $x_{n1}=(g(e_0)-f(e_0))/g(e_1)$,
and $x_{n2}=g(e_1)$, then it follows $f_n\to g$ for
$$f_n=\coAd\left(\,\exp(s_nd)\Phi(0,x_n)\,\right)f$$
in $X$ so that $g\in\overline{X}$. The third case is
$g(e_3)=g(e_1)=0$ and $g(e_2)\neq 0$. If we set
$s_n=n$, $x_{n1}=-e^{n}g(e_2)$, and
$x_{n2}=-e^{-n}(f(e_0)-g(e_0))/g(e_2)$, then it
follows $f_n\to g$. Finally we assume $g(e_{\nu})=0$
for $1\le\nu\le 3$. In this case $s_n=n$,
$x_{n1}=e^{n/2}$, and $x_{n2}=e^{-n/2}(f(e_0)-g(e_0))$
yields $f_n\to g$ so that $g\in\overline{X}$.
\end{proof}

The preceding lemma implies that the set of
critical linear functionals is given by
$\Omega\setminus\overline{X}=\{\,g\in\frakm^\ast:%
g(e_3)=0\text{ and }g(e_1)g(e_2)\neq 0\,\}$. Let us
compute the relevant unitary representations: Using
$\frakp=\langle e_0,e_2,e_3\rangle$ as a Pukanszky
polarization at $f_s\in\frakm^\ast$ for all $s\in\mR$,
one computes that $\pi_s=\mcK(f_s)=\ind_P^M\chi_{f_s}$
in $L^2(\mR)$ is infinitesimally given by
\begin{align*}
d\pi_s(\dot{e}_0)&=f_0+\xi D_{\xi}-i/2,\\
d\pi_s(\dot{e}_1)&=-D_{\xi},\\
d\pi_s(\dot{e}_2)&=e^{-s}\xi,\\
d\pi_s(\dot{e}_3)&=e^{-s}\;.
\end{align*}
Here $\dot{e}_{\nu}=-ie_{\nu}$ is in the
complexification $\frakm_{\mC}$ of $\frakm$,
$\xi\mdot\,-\,$ is the multiplication operator and
$D_{\xi}=-i\partial_{\xi}$ is the differential operator
in $L^2(\mR)$. We observe that these equations bear a
striking resemblance to the formulas for
$\coAd(\Phi(t,x))f\;(e_\nu)$: simply substitute
$e^{-t}x_1$ by  $\xi$ and $e^tx_2$ by $D_\xi$. On the
other hand, if $g\in\Omega\setminus\overline{X}$,
then $\frakn$ is a Pukanszky polarization at
$g\in\frakm^\ast$ and $\rho=\mcK(g)=\ind_N^M\chi_g$
in $L^2(\mR)$ is given by
\begin{align*}
d\rho(\dot{e}_0)&=-D_{\xi},\\
d\rho(\dot{e}_1)&=e^{\xi}\,g_1,\\
d\rho(\dot{e}_2)&=e^{-\xi}\,g_2,\\
d\rho(\dot{e}_3)&=0.
\end{align*}
Symmetrization gives a linear isomorphism from the
symmetric algebra $\mcS(\frakm_{\mC})=\mcP(\frakm^\ast)$
onto the universal enveloping algebra $\mcU(\frakm_{\mC})$
of $\frakm_{\mC}$, which maps the subspace of
$\Ad(M)$-invariant polynomials onto the center
$Z(\frakm_{\mC})$ of $\mcU(\frakm_{\mC})$,
compare Chapter~3.3 of \cite{CorGre}. Note that
$p\in\mcP(\frakm^\ast)^{\Ad(M)}$ corresponds to
$$W\,=\,\dot{e}_3\dot{e}_0-\frac{1}{2}(\dot{e}_2\dot{e}_1%
+\dot{e}_1\dot{e}_2)-f_0\dot{e}_3\,=\,\dot{e}_3\dot{e}_0-%
\dot{e}_2\dot{e}_1-(f_0-\frac{i}{2})\dot{e}_3$$
in $Z(\frakm_{\mC})$. One verifies easily that
$d\tau(W)=p(h)$ holds for all $h\in\frakm^\ast$ and
$\tau=\mcK(h)$. For the Lie algebra $\frakm$ under
consideration the symmetrization map coincides with
the so-called Duflo isomorphism so that $d\tau(W)=p(h)$
can also be seen as a consequence of Th\'eor\`eme~2
of~\cite{Duflo1}.\\\\
Furthermore we recall that if
$\lambda(m)a\,(y)=a(m^{-1}y)$ denotes the left
regular representation of $M$ in $L^2(M)$, then
$$d\lambda(X)a\,(y)=\frac{d}{dt}_{|t=0}\;a(\exp(-tX)y)$$
defines a representation of $\frakm$ in $\Ccinfty{M}$,
which extends to $\mcU(\frakm_{\mC})$. Note that
$\mcU(\frakm_{\mC})$ acts as an associative algebra
of right invariant vector fields. Let us write
$V\ast a=d\lambda(V)a$ for $V\in\mcU(\frakm_{\mC})$
and $a\in\Ccinfty{M}$. It is known that
$\tau(V\ast a)=d\tau(V)\tau(a)$ holds for all $V,a$
and all unitary representations $\tau$ of $M$.

\begin{lem}\label{eLg_lem:b5_regular}
If $g\in\Omega\setminus\overline{X}$ and
$\rho=\mcK(g)$, then $\bigcap_{s\in\mR}\ker_{\LM}\pi_s%
\not\subset\ker_{\LM}\rho$. In particular $G$ is
primitive $\ast$-regular.
\end{lem}
\begin{proof}
Since $\Ccinfty{M}$ is dense in $\LM$, there exists
a function $a\in\Ccinfty{M}$ such that $\rho(a)\neq 0$.
Now $b=W\ast a$ satisfies $\pi_s(b)=d\pi_s(W)\pi_s(a)=%
p(f_s)\pi_s(a)=0$ for all $s\in\mR$ and
$\rho(b)=d\rho(W)\rho(a)=p(g)\rho(a)\neq 0$
because $g\not\in\overline{X}$.
\end{proof}

A priori this result does not seem unlikely because
the nature of $X$ is essentially different from that
of typical non-$\,\ast\,$-regular subsets of
$\frakm^\ast/\coAd(M)$. In the preceding lemma
$X/\coAd(M)$ is a graph over $\frakz\frakm^\ast$ in
the sense that the orbit $\coAd(M)h$ is uniquely
determined by $h\,|\,\frakz\frakm$ for all $h\in X$.
Whereas basic examples of non-$\,\ast\,$-regular
subsets $X$ consist of linear functionals
$h\in\frakm^\ast$ over a common character
$\zeta=h\,|\,\frakz\frakm$ of the center such that
the set of limit points of $X/\coAd(M)$ in
$\frakm^\ast/\coAd(M)$ is not empty.\\\\
Since $\frakg=\frakb_5$ is the only exponential
Lie algebra in dimension $\le 5$ such that there
exist $f\in\frakg^\ast$ in general position with
non-nilpotent, proper stabilizer and critical
functionals $g\in\frakg^\ast$ \wrt the orbit $\coAd(G)f$,
it follows from Lemma~\ref{eLg_lem:b5_regular} that
all exponential Lie groups up to dimension~5 are
primitive $\ast$-regular.\\\\
Note that in the particular case $\frakg=\frakb_5$
the relation $\cap_s\ker_{\mcU(\frakm_{\mC})}\pi_s\not%
\subset\ker_{\mcU(\frakm_{\mC})}\rho$ implies
$\cap_s\ker_{\LM}\pi_s\not\subset\ker_{\LM}\rho$,
but in general, as one might expect, the features
of the universal enveloping algebra do not suffice
for this purpose. However, we anticipate that
$\Ad(M)$-invariant polynomials $p$ corresponding
to elements $W\in Z(\frakm_{\mC})$ will play an
important role in further investigations of
primitive $\ast$-regularity.\\\\
\textbf{Acknowledgment.} The author would
like to thank D.~Poguntke for suggesting the
possibility of proving Proposition~\ref{eLg_prop:%
closed_orbit_is_sufficient} by means of the results
in~\cite{Pog4}. This article owes a lot to his
valuable remarks and comments.

{\small
\bibliographystyle{plain}
\bibliography{literature}
}

\noindent O.~Ungermann\\
Fakult\"at f\"ur Mathematik\\
Universit\"at Bielefeld\\
Postfach 100131\\
D-33501 Bielefeld\\
Germany\\
oungerma@math.uni-bielefeld.de

\end{document}